\documentclass[preprint]{imsart}

\usepackage{amsthm,amsmath}
\RequirePackage{natbib}
\RequirePackage[colorlinks,citecolor=blue,urlcolor=blue]{hyperref}

\doi{10.1214/154957804100000000}
\pubyear{0000}
\volume{0}
\firstpage{0}
\lastpage{0}

\startlocaldefs

\usepackage{amssymb,bbm,graphicx,mathtools,enumitem,caption,subcaption}
\newtheorem{theorem}{Theorem}[section]
\newtheorem{lemma}{Lemma}[section]
\newtheorem{proposition}{Proposition}[section]
\newtheorem{corollary}{Corollary}[section]

\theoremstyle{definition}

\theoremstyle{remark}
\newtheorem{remark}{Remark}[section]

\makeatletter
\newtheorem*{rep@theorem}{\rep@title}
\newcommand{\newreptheorem}[2]{%
\newenvironment{rep#1}[1]{%
 \def\rep@title{#2 \ref{##1}}%
 \begin{rep@theorem}}%
 {\end{rep@theorem}}}
\makeatother

\newreptheorem{theorem}{Theorem}
\newreptheorem{lemma}{Lemma}
\newreptheorem{corollary}{Corollary}

\numberwithin{equation}{section}

\newcommand{\one}[1]{{\mathbbm{1}}_{{#1}}}

\newcommand{\eps}{\epsilon}

\newcommand{\norm}[1]{\lVert{#1}\rVert}
\newcommand{\Norm}[1]{\left\lVert{#1}\right\rVert}

\newcommand{\EE}[1]{\mathbb{E}\left[{#1}\right]}
\newcommand{\Ep}[2]{\mathbb{E}_{#1}\left[{#2}\right]}
\renewcommand{\O}[1]{\mathbf{O}\left({#1}\right)}

\def\R{\mathbb{R}}

\DeclareMathOperator{\nei}{ne}

\DeclareMathOperator{\supp}{supp}
\DeclareMathOperator{\pr}{Prob}

\DeclareMathOperator{\var}{Var}

\newcommand{\bb}{\mathrm{b}}

\renewcommand{\ln}{\log\mathrm{L}} 

\newcommand\wh{\widehat}

\endlocaldefs

\begin{document}

\begin{frontmatter}

\title{High-dimensional Ising model selection with Bayesian
  information criteria} 
\runtitle{Ising model selection}

\author{\fnms{Rina} \snm{Foygel Barber}\ead[label=e1]{rina@uchicago.edu}}
\address{Department of Statistics\\ The University of Chicago\\
  Chicago, IL 60637, U.S.A.\\ \printead{e1}}
\and
\author{\fnms{Mathias} \snm{Drton}\ead[label=e2]{md5@uw.edu}}
\address{Department of Statistics\\ University of Washington\\
  Seattle, WA 98195, U.S.A.\\ \printead{e2}}

\runauthor{R.F.~Barber, M.~Drton}

\begin{abstract}
  We consider the use of Bayesian information criteria for selection
  of the graph underlying an Ising model.  In an Ising model, the full
  conditional distributions of each variable form logistic regression
  models, and variable selection techniques for regression allow one
  to identify the neighborhood of each node and, thus, the entire
  graph.  We prove high-dimensional consistency results for this
  pseudo-likelihood approach to graph selection when using Bayesian
  information criteria for the variable selection problems in the
  logistic regressions.  The results pertain to scenarios of sparsity,
  and following related prior work the information criteria we
  consider incorporate an explicit prior that encourages sparsity.
\end{abstract}





\end{frontmatter}

\section{Introduction}
\label{sec:introduction}

Let $Z_1,\dots,Z_p$ be binary random variables with values in
$\{-1,1\}$, and let $G=(V,E)$ be an undirected graph with vertex set
$V=[p]:=\{1,\dots,p\}$ and edge set $E$ whose elements are unordered
pairs of distinct vertices that we denote by a set of two nodes
$\{v,w\}$.  The (symmetric) Ising model associated to $G$ postulates
that
\begin{equation}
  \label{eq:ising}
  \pr(Z_1=z_1,\dots,Z_p=z_p) \propto \exp\big\{\textstyle\sum_{\{v,w\}\in E}
  \theta_{vw} z_vz_w\big\},
\end{equation}
for values $z_1,\dots,z_p\in\{-1,1\}$ and interaction parameters
$\theta_{ij}\in\mathbb{R}$.  The Ising model is a special case of more
general graphical log-linear or Markov random field models
\citep{lauritzen:1996} but it is of importance in its own right; see
e.g.~\cite{roudi:2009} or the monograph of \cite{kindermann:1980}.  In
this paper we will treat the problem of selecting the graph $G$ based
on a random sample drawn from a distribution in such an Ising model,
complementing recent work on this problem by
\cite{anandkumar:etal:2012}, \cite{ravikumar:2010},
\cite{santhanam:wainwright:2012} and \cite{loh:wainwright:2014}.

The model selection procedure we consider uses a pseudo-likelihood
approach based on conditional distributions, as popularized by
\cite{besag:1972,besag:1974}.  Let
\[
\nei(v) = \{ w\in V\setminus\{v\} \::\: \{v,w\}\in E\}
\]
be the set of neighbors of node $v$ in the graph $G=(V,E)$.
Assuming~(\ref{eq:ising}), the full conditional distributions satisfy
\begin{equation}
  \label{eq:ising-logits}
  \log\left(\frac{\pr(Z_v=1\,|\,Z_w=z_w \ \forall \ w\not=v)}{1-
      \pr(Z_v=1\,|\,Z_w=z_w \ \forall \ w\not=v)} \right)=
  \sum_{w\in\nei(v)} \beta_{vw} z_w,
\end{equation}
where $\beta_{vw}=2\theta_{vw}$.  Hence, for each variable $Z_v$,
the conditional distributions form a logistic regression
model with $Z_v$ as response and the remaining variables $Z_w$ for all
$w\not=v$ as covariates.  Selection of the graph $G=(V,E)$ can thus
be achieved by identifying each neighborhood $\nei(v)$ by
variable selection in each of the $p=|V|$ logistic regression problems
given by~(\ref{eq:ising-logits}).

Strictly speaking, we have $\beta_{vw}=\beta_{wv}$ in the system of
logistic regression models in~(\ref{eq:ising-logits}).  However, we
will treat the neighborhood selection approach in the version that
uncouples the parameters, that is, we allow the pair
$(\beta_{vw},\beta_{wv})$ to range freely in $\mathbb{R}^2$.  This
allows one to treat the $p$ regression problems separately, which
brings about simplifications with regards to computation as well as
theoretical analysis; compare the work on $\ell_1$-penalization
methods by \cite{ravikumar:2010} and by \cite{meinshausen:2006} who
treat the Gaussian case.  \cite{hofling:2009} demonstrated empirically
that this decoupling of $\beta_{vw}$ and $\beta_{wv}$, when addressing
inferential inconsistencies as described in
Section~\ref{sec:pract-cons} below, does not lead to any important
loss in statistical efficiency for selection of the graph $G$ in an
Ising model (at least in the higher-dimensional settings that these
authors and also we have in mind here).
  \cite{hofling:2009} also
showed that, for selection of the graph underlying an Ising model,
pseudo-likelihood methods fare as well as computationally more
involved methods based on the actual joint distribution.  We remark
that while we focus on $\ell_1$-penalization techniques in our later
numerical experiments, the problem of recovering the edges of $G$ in a
high-dimensional setting can also be solved by greedy search methods
\citep{jalali2011learning}.

In this paper, we explore the use of Bayesian information criteria in
the logistic neighborhood selection approach.  Consider a logistic
regression model that includes a subset $J$ of a set of $p$
covariates.  For sample size $n$, and defined for minimization, the
classical Bayesian information criterion (BIC) of \cite{schwarz:1978}
is the model score
\[
\mathrm{BIC}_0(J) = -2\log L(\hat\beta_J) + |J| \log(n),
\]
where $\hat\beta_J$ is the maximum likelihood estimator in the model
given by $J$.  The BIC is well-known to yield variable selection
consistency in the asymptotic scenario in which the sample size $n$
grows large while the number of covariates $p$ remains constant.  It
has been observed, however, that the BIC tends to overselect variables
in regression problems in which $p$ is of substantial size compared to
$n$ \citep{broman:2002}.  To address this problem, a number of
extensions have been proposed and analyzed
\citep{bogdan:2004,chen:chen:2008,chen:chen:2012, frommlet:2012}.  The
main idea for these extensions is to incorporate into the BIC
an explicit prior on the set of considered models.  The priors
specified in the mentioned earlier work are equivalent for
our purposes, as shown in \cite{zak-szatkowska:2011}. Following
  \cite{zak-szatkowska:2011}, we will treat the criterion
\begin{equation}
  \label{eq:ebic}
  \mathrm{BIC}_\gamma(J) = -2\log L(\hat\beta_J) + |J| \big( \log(n)
    + 2\gamma\log(p)\big),
\end{equation}
which is associated with a choice of $\gamma\ge 0$.  For a review and
pointers to prior work that suggests and evaluates defaults for
$\gamma$, or a quantity corresponding to $\gamma$; see
\cite{zak-szatkowska:2011}.  In particular, the choice of $\gamma=1$
is associated with assigning equal prior probability to each set
\[
\mathcal{J}_k=
\{ J\subset [p] \::\: |J|=k\}, \qquad k=0,\dots,q,
\]
where $q$ is an a priori bound on the size of the models; therefore
for each $k\leq q$, any given
 model of size $k$ has probability proportional to $1/|\mathcal{J}_k|$
of being chosen.  The
connection to this prior, which is also considered in
\cite{scott:berger:2010}, is due to the fact that 
\[
|\mathcal{J}_k| =\binom{p}{k}
\]
scales as $p^k$ for small $k\le q\le p/2$.  In~(\ref{eq:ebic}), this
contribution of the prior on models appears as the term $|J|\log(p)$.
Note that~(\ref{eq:ebic}) has the maximum of the log-likelihood
function multiplied by two and, hence, the additional factor of two.
This justifies the criterion \eqref{eq:ebic} for model selection in 
regression.

Now we turn back to the graphical model setting.
By analogy, the prior for Ising model selection has to be specified on
the set of graphs with $p$ nodes and there are
\[
\binom{\binom{p}{2}}{k} \sim p^{2k}
\]
graphs with $k$ edges.  This suggests that for Ising model selection,
$\gamma$ should be chosen roughly twice as large as for variable
selection in a single logistic regression model.  The cutoffs for
$\gamma$ that appear in our theoretical analysis are in agreement with
this intuition (compare Corollary~\ref{cor:ebic-consist-poly} and
Theorem~\ref{thm:cons-ebic-ising}.)

In this paper we show that using $\mathrm{BIC}_\gamma$ for variable
selection in the logistic neighborhood selection approach allows one
to consistently estimate the graph of an Ising model.  Our focus is on
higher-dimensional problems under sparsity, that is, problems in which
the number of variables $p$ may be large, the sample size $n$ may be
comparatively moderate, but the neighborhood sizes are bounded by an
integer $q$ that is small compared to $p$.  Briefly put, under the
conditions we impose, $\mathrm{BIC}_\gamma$ can successfully identify
the graph if $n$ exceeds a constant multiple of $q^3\log(p)$, which
agrees with the rates found in \cite{ravikumar:2010} and
\cite{santhanam:wainwright:2012}.

Our work builds on ideas of \cite{chen:chen:2012}
and \cite{luo:chen:2013:sii} who analyze the performance of
$\mathrm{BIC}_\gamma$ for variable selection in generalized linear
models.  Their work makes assumptions on a sequence of
fixed/deterministic design matrices that ensure that the Hessian of
the log-likelihood function is well-behaved.  In contrast, the
conditional distributions in~(\ref{eq:ising-logits}) have random
covariates.  We thus develop suitable conditions on the joint
distribution of random covariates in logistic regression that, in
particular, ensure that the deterministic conditions imposed in
\cite{luo:chen:2013:sii} hold with high probability.  
The conditions we give allow us to deduce
consistency of $\mathrm{BIC}_\gamma$ in Ising model selection.  For
growing $p$, this involves a growing number of logistic regression
problems and requires us to make some of the intermediate results in
\cite{luo:chen:2013:sii} more explicit.

The paper is organized as follows.  Section~\ref{sec:logist-reg}
provides finite-sample results for logistic regression.  The main
technical result is Theorem~\ref{thm:RandomToHessian}, which considers
the setting with random covariates and gives conditions that provide
control of the Hessian of the log-likelihood function.
Theorem~\ref{thm:HessianToLikelihood} shows how a well-behaved Hessian
leads to bounds on likelihood ratios and is closely related to the
prior work of \cite{chen:chen:2012} and \cite{luo:chen:2013:sii}.  The
proofs for both these theorems are deferred to
parts~\ref{sec:LikelihoodAndScore} and~\ref{sec:Hessian} of the
Appendix, where part~\ref{sec:technical-lemmas} contains technical
lemmas.  As a consequence of Theorems~\ref{thm:RandomToHessian}
and~\ref{thm:HessianToLikelihood}, we can clarify in
Section~\ref{sec:cons-ebic-logist} the consistency of
$\mathrm{BIC}_\gamma$ in logistic regression with random covariates.
In Section~\ref{sec:cons-ebic-ising}, we extend the consistency result
to Ising models.  Some of the conditions imposed in our work
involve third moments, and we show in part~\ref{sec:why-are-second} of
the Appendix that those cannot be weakened to conditions on second
moments.  We conclude with numerical experiments on simulated and real data, see
Sections~\ref{sec:pract-cons} and~\ref{sec:experiments}, and a
discussion in Section~\ref{sec:discussion}.

\section{Logistic regression with random covariates}
\label{sec:logist-reg}

\subsection{Setup}

Let $(X_1,Y_1),\dots,(X_n,Y_n)$ be $n$ observations that each pair a
binary response $Y_i\in\{0,1\}$ and a covariate vector $X_i\in\R^p$.
Suppose that the pairs $(X_i,Y_i)$ are independent and identically
distributed, and that the responses follow a logistic regression model
conditional on the covariates.  Let $\pi_i(x)$ be the conditional
probability that $Y_i=1$ given $X_i=x$.  The logistic regression
model states that
\[
\log\left( \frac{\pi_i(x)}{1-\pi_i(x)}\right)= x^\top\beta_0
\]
for some unknown parameter vector $\beta_0\in\R^p$.  Define the
cumulant function $\mathrm{b}(\theta)=\log(1+e^{\theta})$.
Conditional on the $X_i$, the logistic regression model for the
responses $Y_i$ has log-likelihood, score, and negative Hessian
functions
\begin{align*}
  \ln(\beta)&=\sum_{i=1}^n Y_i\cdot X_i^{\top}\beta - 
  \mathrm{b}(X_i^{\top}\beta)\;\in\;\R\;,\\
  s(\beta)&=\sum_{i=1}^n X_i \left(Y_i - 
    \mathrm{b}'(X_i^{\top}\beta)\right)\;\in\;\R^p\;,\\
  H(\beta)&=\sum_{i=1}^n X_iX_i^{\top}\cdot 
  \mathrm{b}''(X_i^{\top}\beta)\;\in\;\R^{p\times p}\;,
\end{align*}
with the derivatives of the cumulant function being
\begin{equation}
  \label{eq:b':b''}
  \mathrm{b}'(\theta) = \frac{e^\theta}{1+e^\theta}, \quad
  \mathrm{b}''(\theta) = \frac{e^\theta}{\left(1+e^\theta\right)^2}.
\end{equation}

We will be interested in scenarios in which $\beta_0$ is sparse, and we
wish to recover the support of $\beta_0$, that is, the set
\[
J_0=\supp(\beta_0):=\{j\in[p]\::\: \beta_{0j}\not= 0\},
\]
which gives the most parsimonious (most sparse) true model.
We assume that an upper bound $q$ on the size of the support is given, that
is, $|J_0|\leq q$.  
Later, the bound $q$ is allowed to grow in an asymptotic scenario in
which the number of covariates $p$ may grow with the sample size $n$.
To avoid triviality, we assume $n,p\geq 2$ throughout.  Similarly, we
assume $q\ge 1$ without further mention.

The conditions we impose below are formulated in terms of the marginal
distribution of the covariate vectors $X_i$ and
pertain to the tail behavior of the entries of $X_i$ as well as the
possible dependences among them.  We will show that our conditions
entail that, with large probability, the covariates satisfy
deterministic Hessian conditions that \cite{luo:chen:2013:sii} used to
establish consistency properties of $\mathrm{BIC}_\gamma$ for
generalized linear models with fixed design.  These conditions concern
sparse submodels of our logistic regression model given by support
sets $J\subseteq[p]$.  

\paragraph{Notation for submodels.}   
The parameters of the submodel given by a set $J$ are regression
coefficients that form a vector of length $|J|$.  We index such
vectors $\beta$ by the elements of $J$, that is, $\beta=(\beta_j:j\in
J)$, and similarly write $\R^J$ for the parameter space comprising all
these coefficient vectors.  This way the index of a coefficient always
coincides with the index of the covariate it belongs to.
In other words, the coefficient for the $j$-th coordinate of covariate vector $X_i$ is denoted by $\beta_j$ in any model $J$ with $j\in J$.

Furthermore, it is at times convenient to identify a vector
$\beta\in\R^J$ with the vector in $\R^p$ that is obtained from $\beta$
by filling in zeros outside of the set $J$.  As this is clear from the
context, we simply write $\beta$ again when referring to this sparse
vector in $\R^p$.  Finally, $s_J(\beta)$ and $H_J(\beta)$ denote
the subvector and submatrix of $s(\beta)$ and $H(\beta)$,
respectively, obtained by extracting entries indexed by $J$.

\subsection{Hessian conditions when covariates are random}

\cite{luo:chen:2013:sii} 
invoke conditions on a sequence of deterministic designs to control
the curvature and change of the Hessian of the log-likelihood
function.  Specifically, the eigenvalues of $\frac{1}{n}H_J(\beta_0)$
for all sparse $J\supseteq J_0$ are assumed to be bounded above and
below, and furthermore for any $\eps$, there is a $\delta>0$ such that
\begin{equation}
  \label{LuoChenBoundedChangeAssump}
  (1-\eps)H_J(\beta_0)\preceq 
  H_J(\beta)\preceq (1+\eps)H_J(\beta_0)\;,
\end{equation}
for all $J\supseteq J_0$ and $\beta\in\R^J$ with
$\norm{\beta-\beta_0}_2\leq \delta$.  The notation ``$\preceq$''
refers to the ordering in the positive semidefinite cone with
$A\preceq B$ whenever $0\preceq B-A$, i.e., $B-A$ is positive
semidefinite.  The above conditions are assumed to hold uniformly
for all large enough sample sizes $n$ and associated values of $p$,
$q$ and $\beta_0$, which may change with $n$.

In this work, we begin instead with random and i.i.d.~covariates
$X_1,\dots,X_n$ and derive stronger versions of these Hessian
conditions from the below conditions on the distribution of each
covariate $X_i$.
We refer to a vector $u\in\R^p$ as $q$-sparse if $|\supp(u)|\le q$.
Let $a_1,a_2,a_3>0$ be constants that are fixed throughout the
remainder of this section.  Using $X_1=(X_{11},\dots,X_{1p})^\top$ as
a representative, we will say that the i.i.d.~covariates
satisfy assumptions \ref{ass:A1}-\ref{ass:A3} with respect to an
integer $q\ge 1$ if the following holds:

\begin{enumerate}[label=(A\arabic*)]
\item \label{ass:A1} For any $q$-sparse unit vector $u$, $\EE{(X_1^{\top}u)^2}\geq a_1$. 
\item \label{ass:A2} For any $q$-sparse unit vector $u$, $\EE{|X_1^{\top}u|^3}\leq a_2$.
\item \label{ass:A3} For each $j\in[p]$, the variable $X_{1j}$ is
  bounded 
  as $|X_{1j}|\leq a_3$.
\end{enumerate}

Rephrased, \ref{ass:A1} states that for any subset $J\subset[p]$ of
cardinality $|J|\leq q$ the smallest eigenvalue of the 
matrix $\EE{X_{1J}X_{1J}^{\top}}$ is at least $a_1$. (Here,
$X_{1J}=(X_{1j}:j\in J)$ is the subvector of $X_1$ induced by $J$.)
Assumption~\ref{ass:A2} guarantees the existence of third moments of
linear combinations of $q$ or fewer covariates.  In an Ising model all
variables are bounded and thus \ref{ass:A3} always holds.\footnote{A
  weaker condition requiring only that each $X_{1j}$ is subgaussian
  was considered in a preprint version of this paper
  \citep{oldversion}.  The same results were obtained, but at the cost
  of additional log factors in the sample size---specifically, with a
  sample size requirement of $n\gtrsim q^3\log^3(np)$ instead of
  $n\gtrsim q^3\log(p)$ as in the theorems  in this paper.}

According to the following theorem, our assumptions entail
well-behaved Hessians with large probability.  In
this theorem and throughout the rest of the paper, the norm $\|H\|$ of
a matrix $H$ is the spectral norm.

\begin{theorem}\label{thm:RandomToHessian}
Suppose that the covariates satisfy conditions \ref{ass:A1}-\ref{ass:A3} for
some sparsity level $q$ and some constants $a_1,a_2,a_3>0$. Then
  there exist constants
  $c_{\mathrm{sample}},c_{\mathrm{change}},c_{\mathrm{prob}}>0$, a
  decreasing function $c_{\mathrm{lower}}:[0,\infty)\rightarrow
  (0,\infty)$ and an increasing function
  $c_{\mathrm{upper}}:[0,\infty)\rightarrow (0,\infty)$, all depending
  only on $(a_1,a_2,a_3)$, such that
  if 
  \[
  n \ge 
  c_{\mathrm{sample}}\cdot q^3\log(p),
  \]
  then the event that, simultaneously for all $|J|\leq q$ and all
  $\beta,\beta'\in\R^J$,
  \begin{equation}
    \label{Hess1}
    c_{\mathrm{lower}}(\norm{\beta}_2)\mathbf{I}_J\preceq 
    \frac{1}{n}H_J(\beta)\preceq 
    c_{\mathrm{upper}}(\norm{\beta}_2)\mathbf{I}_J
  \end{equation}
  and
  \begin{equation}
    \label{Hess2}
    \frac{1}{n}\norm{H_J(\beta) - H_J(\beta')} \leq c_{\mathrm{change}}\cdot
    \norm{\beta-\beta'}_2\;
  \end{equation}
  has probability at least  
  \[
  1-\exp\left\{-\,c_{\mathrm{prob}}\cdot \frac{n}{q^3}\right\}.
  \]
\end{theorem}
\noindent The proof of Theorem~\ref{thm:RandomToHessian} is given in Appendix~\ref{sec:Hessian}.

If the inequalities~(\ref{Hess1}) and~(\ref{Hess2}) hold and
$\beta\in\mathbb{R}^J$ for a set $J\supseteq J_0$, then 
\begin{align*}
  \frac{1}{n} H_J(\beta) &\preceq c_{\mathrm{change}}\cdot
  \|\beta-\beta_0\|_2 \cdot\mathbf{I}_J + \frac{1}{n} H_J(\beta_0) \\
  &\preceq \left(1+\frac{c_{\mathrm{change}}}{c_{\mathrm{lower}}(\|\beta_0\|_2)}\cdot
    \|\beta-\beta_0\|_2 \right) \frac{1}{n} H_J(\beta_0)\;.
\end{align*}
We also have the analogous lower bound, 
\begin{align*}
  \frac{1}{n} H_J(\beta) &\succeq
  \left(1-\frac{c_{\mathrm{change}}}{c_{\mathrm{lower}}(\|\beta_0\|_2)}\cdot
    \|\beta-\beta_0\|_2 \right) \frac{1}{n} H_J(\beta_0) \;.
\end{align*}
Combining these two bounds, we have proved the following version of the assumption
from~(\ref{LuoChenBoundedChangeAssump}):

\begin{proposition}
  \label{prop:hess:in:luo:chen:form}
  If the inequalities~(\ref{Hess1}) and~(\ref{Hess2}) hold, then 
  \[
  (1-\eps)H_J(\beta_0)\preceq 
  H_J(\beta)\preceq (1+\eps)H_J(\beta_0)\;
  \]
  for all $J\supseteq J_0$ and $\beta\in\mathbb{R}^J$ with 
  \begin{equation}
    \label{eq:luo:chen:delta}
    \|\beta-\beta_0\|_2\le 
    \delta := \epsilon\cdot \frac{
      c_{\mathrm{lower}}(\|\beta_0\|_2)}{c_{\mathrm{change}}}. 
  \end{equation}
\end{proposition}
\begin{remark}
  Although this proposition only treats true models (i.e., models $J$
  that contain the true support $J_0$), it will be used also for proving
  that the BIC will not select a false model (i.e., a model
  $J\not\supset J_0$).  The connection lies in observing that, for a
  model $J\not\supset J_0$, the proposition can be applied to analyze
  the model given by the union $J\cup J_0$, which is a true model.
\end{remark}

\subsection{Bounds on
  likelihood ratios from Hessian conditions}


The following theorem provides bounds on log-likelihood ratios for
sparse models indexed by $J$ versus the smallest true model indexed by
$J_0$.  The result concerns fixed values for the covariates
$X_1,\dots,X_n$ that satisfy the Hessian conditions~(\ref{Hess1})
and~(\ref{Hess2}) from Theorem~\ref{thm:RandomToHessian}.  The
statement of the result makes reference to constants from
Theorem~\ref{thm:RandomToHessian}.  We also invoke an upper bound
$a_0$ on the signal; some control of the norm of $\beta_0$ is needed
to avoid degeneracy of the conditional distribution of the binary
response variable.


\begin{theorem}\label{thm:HessianToLikelihood}
  Let $\beta_0$ be the true parameter with $J_0=\supp(\beta_0)$ and
  $\norm{\beta_0}_2\leq a_0$ for a constant $a_0>0$.  Fix
  $\eps,\nu>0$, and condition on the covariates $X_1,\dots,X_n$ 
  satisfying the Hessian conditions~(\ref{Hess1})
  and~(\ref{Hess2}) for all $J\supseteq J_0$ with $|J|\le 2q$, where
  $q\ge |J_0|$.  Then there exist constants
  $C_\mathrm{false},C_\mathrm{dim},C_{\mathrm{sample},1},C_{\mathrm{sample},2}>0$,
  depending only on $(c_{\mathrm{change}},
  c_{\mathrm{lower}}(a_0),c_{\mathrm{upper}}(a_0))$ and on the chosen
  pair $(\eps,\nu)$, such that if
  \[
  p \ge C_\mathrm{dim}\quad\text{and}\quad n \ge \max\left\{ C_{\mathrm{sample},1}\cdot q^3\log(p)
  , C_{\mathrm{sample},2}\cdot \frac{q\log(p)}{\min_{j\in J_0}|(\beta_0)_j|^2}\right\} ,
  \] 
  the following two statements hold simultaneously with conditional
  probability at least $1-p^{-\nu}$:
  \begin{enumerate}[label=(\alph*)]
  \item \label{thm2:a} For all $|J|\leq q$ with
    $J\supseteq J_0$,
    \[
    \ln(\wh{\beta}_J)-\ln(\wh{\beta}_{J_0})\;\leq\;
    (1+\eps)(|J\backslash J_0|+\nu)\log(p)\;.
    \]
  \item \label{thm2:b} For all $|J|\leq q$ with
    $J\not\supset J_0$,
    \[
    \ln(\wh{\beta}_{J_0})-\ln(\wh{\beta}_J)\;\geq\;
    C_\mathrm{false} \,n\min_{j\in J_0}|(\beta_0)_j|^2\;.
    \]
  \end{enumerate}
\end{theorem}

The proof of Theorem~\ref{thm:HessianToLikelihood} is deferred to
Appendix~\ref{sec:LikelihoodAndScore}.  We remark that the proof of
claim~\ref{thm2:a} invokes the Hessian conditions only for $J\supseteq
J_0$ with $|J|\le q$.  The conditions for cardinality up to $2q$ are
used for claim~\ref{thm2:b}, which is proved by considering the union
$J_0\cup J$ for the given false model $J\not\supset J_0$.

\subsection{Consistency of extended BIC  in logistic regression}
\label{sec:cons-ebic-logist}

Having established bounds on Hessian and likelihood ratios via
Theorem~\ref{thm:RandomToHessian} and
Theorem~\ref{thm:HessianToLikelihood}, respectively, we are able to
give conditions that entail that $\mathrm{BIC}_\gamma$ selects the
most parsimonious true model with high probability.

\begin{theorem}
  \label{thm:logistic:ebic}
  Let $\beta_0$ be the true parameter with $J_0=\supp(\beta_0)$ and
  $\norm{\beta_0}_2\leq a_0$ for a constant $a_0>0$.  Fix $\gamma\ge
  0$ and $\epsilon,\nu>0$.  Then there exist constants
  $C_0,C_1,C_2,C_3>0$, depending only on $(a_0,a_1,a_2,a_3)$ and
  $(\epsilon,\nu)$, such that if the covariates satisfy
  \ref{ass:A1}-\ref{ass:A3} with respect to $2 q$ for $q \ge |J_0|$,
  if
\[
    p\ge C_0, \quad n \ge \max\left\{ 
    C_1\cdot q^3\log(p), \,C_2\cdot \frac{q\log(n p^{2\gamma})}{\min_{j\in
        J_0}|(\beta_0)_j|^2}\right\},
\]
  and if 
  \begin{equation}
    \label{eq:logistic:choice:gamma}
  \sqrt{n}>p^{(1+\epsilon)(1+\nu)-\gamma},
  \end{equation}
  then the event that
  \[
  J_0 = \arg\min \{ \mathrm{BIC}_\gamma(J) \::\: J\subset[p],\, |J|\le q\}
  \]
  has probability at least
  \[
  \left(1-\exp\left\{-\,C_3\cdot
      \frac{n}{q^3}\right\} \right)\left(
    1-\frac{1}{p^{\nu}}\right).
  \]
\end{theorem}
\begin{proof}
 First, examining the statements of  Theorem~\ref{thm:RandomToHessian} and
  Theorem~\ref{thm:HessianToLikelihood}, we see that we
  can choose the constants $C_0,C_1,C_2,C_3$ large enough that the 
  conditions in 
  Theorems~\ref{thm:RandomToHessian} and~\ref{thm:HessianToLikelihood} 
  are satisfied. These theorems then imply that, with the claimed probability,
  the following statement is true simultaneously for all $|J|\le q$:
  \begin{equation}
    \label{eq:lr-bounds:ebic}
    \ln(\wh{\beta}_J)-\ln(\wh{\beta}_{J_0})\leq
    \begin{cases}
      (1+\eps)(|J\backslash J_0|+\nu)\log(p) &\text{if }\ J\supseteq
      J_0,\\
      -C_\mathrm{false} n\min_{j\in J_0}|(\beta_0)_j|^2 &\text{if }\ J\not\supseteq J_0,
    \end{cases}
  \end{equation}
  where $C_\mathrm{false}>0$ is a constant from
  Theorem~\ref{thm:HessianToLikelihood}.  Condition
  on~(\ref{eq:lr-bounds:ebic}) being true for all $|J|\le q$.  We
  claim that under our assumptions
  \begin{multline*}
    \mathrm{BIC}_{\gamma}(J)-\mathrm{BIC}_{\gamma}(J_0)
    =-2\left(\ln(\wh{\beta}_J)-\ln(\wh{\beta}_{J_0})\right)\\
    + (
    |J|-|J_0|)\big(\log(n)+2\gamma\log(p)\big) 
  \end{multline*}
  is positive for any model given by a set $J\neq J_0$ of cardinality $|J|\le
  q$.

  If $J\not\supseteq J_0$, that is, if the model is false,
  then~(\ref{eq:lr-bounds:ebic}) yields the bound
  \begin{equation*}
    \mathrm{BIC}_{\gamma}(J)-\mathrm{BIC}_{\gamma}(J_0)
    \;\geq\; 2C_\mathrm{false} n\min_{j\in J_0}|(\beta_0)_j|^2-q\log(n p^{2\gamma})\;.
  \end{equation*}
 Since we require that  $n \ge C_2\cdot \frac{q\log(n p^{2\gamma})}{\min_{j\in
        J_0}|(\beta_0)_j|^2}$,  this lower bound
   on $ \mathrm{BIC}_{\gamma}(J)-\mathrm{BIC}_{\gamma}(J_0)$
 is positive for a sufficiently large choice of the constant $C_2$.

  For $J\supsetneq J_0$ with $|J|\leq q$, we have
  \begin{multline*}
    \mathrm{BIC}_{\gamma}(J)-\mathrm{BIC}_{\gamma}(J_0)
    \;\geq\; -2(1+\eps)(|J\backslash J_0|+\nu)\log(p)\\
    + |J\backslash J_0|\left(\log(n)+2\gamma\log(p)\right)\;,
  \end{multline*}
  which can be lower-bounded further as
  \begin{align*}
    \mathrm{BIC}_{\gamma}(J)-\mathrm{BIC}_{\gamma}(J_0) \;\geq\;
    |J\backslash J_0|\cdot \left( \log(n) + 2\big[\gamma
      -(1+\eps)(1+\nu)\big]\log(p) \right)\;.
  \end{align*}
  This is positive by the assumed inequality
  from~(\ref{eq:logistic:choice:gamma}). 
\end{proof}

Based on Theorem~\ref{thm:logistic:ebic}, we can identify asymptotic
scenarios under which $\mathrm{BIC}_\gamma$ yields consistent variable
selection.  To this end, consider a sequence of variable selection
problems indexed by the sample size $n$, where the $n$-th problem has
$p_n$ covariates and true parameter $\beta_0(n)$ with support
$J_0(n)$.  Let $q_n$ be the bound on the size of the considered
models, and let
\[
\beta_{\min}(n) = \min_{j\in
        J_0(n)}|\beta_0(n)_j|
\]
be the smallest absolute value of any non-zero coefficient in
$\beta_0(n)$.

\begin{corollary}
  \label{cor:ebic-consist-poly}
  Suppose that $p_n\rightarrow\infty$ as $n\rightarrow\infty$ with $p_n\leq n^\kappa$ for some $\kappa\in(0,\infty]$ and
 $\log(p_n)\leq n^\tau$ for some $0< \tau<1$.  Suppose further that $q_n\leq n^\psi$ for
  some $0\le \psi<\frac{1}{3}(1-\tau)$, and that $\beta_{\min}(n)\geq n^{-\phi/2}$ for
  some $0\le\phi<1-\psi-\tau$.
  Assume that the covariates satisfy \ref{ass:A1}-\ref{ass:A3} with
  respect to $2 q_n$ for some constants $a_1,a_2,a_3>0$,
  and that $|J_0(n)|\le q_n$, and
 $\norm{\beta_0(n)}_2\le a_0$ for a constant $a_0>0$.  Then for any
  $\gamma>1-\frac{1}{2\kappa}$, variable selection with
  $\mathrm{BIC}_\gamma$ is consistent in the sense that the event
  \[
  J_0(n) = \arg\min \{ \mathrm{BIC}_\gamma(J) \::\: J\subset[p_n],\,
  |J|\le q_n\}
  \]
  has probability tending to one as $n\to\infty$. 
\end{corollary}
\begin{proof}
  Since $p_n\leq n^\kappa$, condition~(\ref{eq:logistic:choice:gamma}) in
  Theorem~\ref{thm:logistic:ebic} holds for all $n$ if
  \[
  \frac{1}{2\kappa} \;>\; (1+\epsilon)(1+\nu)-\gamma.
  \]
  Having assumed $\gamma>1-\frac{1}{2\kappa}$ here, the condition is
  satisfied for $\epsilon$ and $\nu$ sufficiently small.  Fix a
  suitable choice of $(\epsilon,\nu)$ for the rest of the argument.

  Our scaling assumptions for $p_n$, $q_n$ and $\beta_{\min}(n)$ are such that
  the conditions involving the constants $C_0$, $C_1$ and $C_2$ in
  Theorem~\ref{thm:logistic:ebic} are met for $n$ large enough.
  Hence, Theorem~\ref{thm:logistic:ebic} applies for all large $n$.
  And, as $n\to\infty$, the probability in
  Theorem~\ref{thm:logistic:ebic} tends to one.
\end{proof}

\begin{remark}
  Corollary~\ref{cor:ebic-consist-poly} requires $p_n\leq n^{\kappa}$ 
  and $\log(p_n)\leq n^\tau$, for $\kappa\in(0,\infty]$ and $\tau\in(0,1)$.
  For $\kappa<\infty$, this means that $p_n$ grows polynomially with
  $n$.  In this case, $\tau$ can be chosen arbitrarily close to 0, and the conditions on $\psi$ and $\phi$ become $0\le \psi<1/3$ and $0\le \phi<1-\psi$.  For $\kappa=\infty$, the growth of $p_n$ can be faster
  than polynomial; the remaining condition $\log(p_n)\leq n^\tau$ allows
  for subexponential growth. In this latter case, since $\kappa=\infty$, we
  require $\gamma>1$ in order to ensure consistency of $\mathrm{BIC}_\gamma$.
\end{remark}

\section{Consistency of extended BIC for Ising models}
\label{sec:cons-ebic-ising}

Turning to neighborhood selection for Ising models, let
$Z_1,\dots,Z_n$ be an i.i.d.~sample, where each
$Z_i=(Z_{i1},\dots,Z_{ip})$ is a vector of binary random variables
with values in $\{-1,1\}$.  Suppose the $Z_i$ follow an Ising model as
in~(\ref{eq:ising}), with graph $G=(V,E)$ on the vertex set $V=[p]$,
and interaction parameters $\theta_{vw}\in\mathbb{R}$ for $\{v,w\}\in
E$.  Assume that $G$ is minimal in that $\{v,w\}\in E$ if
and only $\theta_{vw}\not=0$.

We will consider selection of $G$ by means of variable selection in
the $p$ logistic regression models, where the $v$-th regression
problem has response variable $Z_v$ and the $p-1$ covariates $Z_w$, 
$w\in [p]\setminus\{v\}$.  We write $\mathrm{BIC}_\gamma(J,v)$ for the
BIC score from~(\ref{eq:ebic}) evaluated for the logistic regression
model with response $Z_v$ and covariates $Z_w$, $w\in J$, with
$J\subseteq [p]\setminus\{v\}$.  Correct inference of $G$ is achieved
if, for each $v\in[p]$, the neighborhood
\[
\nei(v)=\left\{ w\in [p]\setminus\{v\} \::\:
  \theta_{vw}\not=0\right\} =\left\{ w\in [p]\setminus\{v\} \::\:
  \{v,w\}\in E\right\}
\]
(uniquely) minimizes $\mathrm{BIC}_\gamma(\cdot,v)$.

Using $Z_1=(Z_{11},\dots,Z_{1p})^\top$ as a representative, we will
say that $Z_1,\dots,Z_n$ satisfy
assumptions~\ref{ass:ising:2}-\ref{ass:ising:4} with respect to an
integer $q\ge 1$ if the following holds for fixed constants
$b_0,b_1,b_2>0$:

\begin{enumerate}[label=(B\arabic*)]
\item \label{ass:ising:2} The interaction between a variable and its
  neighborhood is bounded as
  \[
  \sqrt{\sum_{w\in\nei(v)}\theta_{vw}^2}\le b_0 \quad\text{for all}\quad v\in
  [p].
  \]
\item \label{ass:ising:3} For any $q$-sparse unit vector $u$,
  $\EE{(Z_1^{\top}u)^2}\geq b_1$.
\item \label{ass:ising:4} For any $q$-sparse unit vector $u$,
  $\EE{|Z_1^{\top}u|^3}\leq b_2$.
\end{enumerate}

As explained in \cite{santhanam:wainwright:2012}, the graph selection
problem is ill-posed without some upper bound on the interaction
between a variable and its neighborhood, as we impose
in~\ref{ass:ising:2}.  Assumption~\ref{ass:ising:3} constitutes a
lower bound on the eigenvalues of the $q\times q$ principal
submatrices of the covariance matrix $\EE{Z_1Z_1^\top}$ and is akin to
requirements in \cite{ravikumar:2010} and \cite{loh:wainwright:2014}.  As we
clarify at the end of this section, condition~\ref{ass:ising:3} is
implied by~\ref{ass:ising:2} for asymptotic scenarios in which all
neighborhoods $\nei(v)$ have cardinality bounded by a constant, that
is, the graph $G$ has bounded degree.  Assumption~\ref{ass:ising:4} is
the final piece needed to invoke our result on general logistic
regression.

To formulate a consistency result for neighbor selection in Ising
models, we consider a sequence of neighborhood selection problems
indexed by the sample size $n$.  The $n$-th problem has $p_n$
variables and interaction parameters $\theta_{vw}(n)$, with associated
neighorhoods $\nei_n(v)$ and edge set $E(n)$.  Let $d_n$ be the
maximum cardinality of any neighborhood $\nei_n(v)$, $v\in[p_n]$,
and let
\[
\theta_{\min}(n)=\min_{\{v,w\}\in E(n)} |\theta_{vw}(n)|
\]
be the non-zero interaction of smallest magnitude.

\begin{theorem}
  \label{thm:cons-ebic-ising}
   Suppose that $p_n\rightarrow\infty$ as $n\rightarrow\infty$ with $p_n\leq n^\kappa$ for some $\kappa\in(0,\infty]$ and
  $\log(p_n)\leq n^\tau$ for some $0< \tau<1$.  Suppose further that $q_n\leq n^\psi$ for
  some $0\le \psi<\frac{1}{3}(1-\tau)$, and that $\theta_{\min}(n)\geq n^{-\phi/2}$ for
  some $0\le\phi<1-\psi-\tau$.
  Assume that the sample $Z_1,\dots,Z_n$ satisfies
  \ref{ass:ising:2}-\ref{ass:ising:4} with respect to $2 q_n$ and that
  $d_n\le q_n$.  Then for any $\gamma>2-\frac{1}{2\kappa}$, Ising
  neighborhood selection with $\mathrm{BIC}_\gamma$ is consistent in
  the sense that the event that, simultaneously for all $v\in[p_n]$,
  \[
  \nei_n(v) = \arg\min \{ \mathrm{BIC}_\gamma(J,v) \::\:
  J\subset[p_n]\setminus\{v\},\, |J|\le q_n\}
  \]
  has probability tending to one as
  $n\to\infty$.
\end{theorem}
\begin{remark} As in Corollary~\ref{cor:ebic-consist-poly},
this result allows for subexponential rather than polynomial
growth of $p_n$ relative to $n$, by setting $\kappa=\infty$.
\end{remark}
\begin{proof}[Proof of Theorem~\ref{thm:cons-ebic-ising}]
  We will show that the result follows from
  Theorem~\ref{thm:logistic:ebic} together with a union bound over the
  $p_n$ logistic regression problems.

  First, we observe that with $p_n\leq n^\kappa$,
  condition~(\ref{eq:logistic:choice:gamma}) in
  Theorem~\ref{thm:logistic:ebic} holds for all $n$ if
  \[
  \frac{1}{2\kappa} \;>\; (1+\epsilon)(1+\nu)-\gamma.
  \]
  Having assumed $\gamma>2-\frac{1}{2\kappa}$ here, the condition can
  be satisfied with a choice of $\epsilon>0$ and $\nu>1$.  We fix such
  a choice of $(\epsilon,\nu)$ for the rest of the argument.

  Next, note that Theorem~\ref{thm:logistic:ebic} is applicable to
  each one of the $p_n$ logistic regression problems in neighborhood
  selection.  Indeed, since $Z_1,\dots,Z_n$ are bounded
  assumption~\ref{ass:A3} holds.  Conditions~\ref{ass:A1}
  and~\ref{ass:A2} are ensured by~\ref{ass:ising:3}
  and~\ref{ass:ising:4}, respectively, and \ref{ass:ising:2} yields
  the bounded signal assumed in Theorem~\ref{thm:logistic:ebic}.
  Moreover, the scaling assumptions on $p_n$, $q_n$ and
  $\theta_{\min}(n)$ are such that the assumptions on the
  corresponding quantities in Theorem~\ref{thm:logistic:ebic} are met.

  Applying Theorem~\ref{thm:logistic:ebic} a total of $p_n$ times, we
  obtain that, separately for each $v\in[p_n]$, the event that
  \[
  \nei_n(v) = \arg\min \{ \mathrm{BIC}_\gamma(J,v) \::\:
  J\subset[p_n]\setminus\{v\},\, |J|\le q_n\}
  \]
  occurs with at least the probability from
  Theorem~\ref{thm:logistic:ebic}.  Ignoring smaller terms of higher
  order in $1/p_n$, this probability is
  \[
  1-  \frac{1}{np_n} - \frac{1}{p_n^\nu}.
  \]
  Since $\nu>1$, we have that
  \[
  p_n\cdot \left(\frac{1}{np_n} + \frac{1}{p_n^\nu} \right)
  \;\longrightarrow \; 0
  \]
  as $n$, and thus also $p_n$, tends to infinity.  Hence, a union
  bound yields the desired claim that all events hold simultaneously
  with probability tending to one.
\end{proof}

Finally, we observe that conditions~\ref{ass:ising:3} and
~\ref{ass:ising:4} do not present a restriction when considering
problems in which there is a fixed bound on the degree of the graph
underlying the Ising model and a bound on the interaction parameters
as in~\ref{ass:ising:2}.  Indeed, \ref{ass:ising:4} holds trivially in
this case since the coordinate of the random vectors are bounded by
one in absolute value.  The sparse eigenvalue
condition~\ref{ass:ising:3} is addressed in the next lemma.

\begin{lemma}
  \label{lem:eig-values-bounded-degree}
  Suppose the random vector $Z=(Z_1,\dots,Z_p)$ follows an Ising model
  with $|\nei(v)|\le q$ for all $v\in [p]$.  If the interaction
  parameters $\theta_{vw}$ for $Z$ satisfy~\ref{ass:ising:2} then it
  holds for any $q$-sparse unit vector $u$ that
  \[
  \EE{(Z^\top u)^2}\;\ge\; \frac{4}{q}\cdot\frac{e^{ 2
        b_0\sqrt{q}}}{\left(1+e^{ 2b_0\sqrt{q}}\right)^2}.
  \]
\end{lemma}
\begin{proof}
  Without loss of generality, we consider a $q$-sparse unit vector $u$
  that has $\supp(u)=\{1,\dots,q\}$ and
  \[
  |u_1|\ge |u_2|\ge \dots \ge |u_q|.
  \]
  Then $u_1^2\ge 1/q$.  Let $Z_{-1}=(Z_2,\dots,Z_p)^\top$.  For a
  random variable $X$ with finite variance,
  \[
  \var[X]=\min_{a\in\mathbb{R}}\EE{(X-a)^2}.
  \]
  Therefore,
  \begin{equation*}
    \EE{(Z^\top u)^2 \,|\, Z_{-1}} \;\ge \; \var\left[ Z_1u_1 \,|\,
      Z_{-1}\right]\;\ge \; \frac{1}{q}\var\left[ Z_1 \,|\,
      Z_{-1}\right]. 
  \end{equation*}
  Since $Z_1$ takes values in $\{-1,1\}$, we rescale to $(Z_1+1)/2$
  for values in $\{0,1\}$.  Then the conditional distribution of
  $(Z_1+1)/2$ given $Z_{-1}$ is a Bernoulli distribution with success
  probability
  \[
  \frac{\exp\left\{ 2\sum_{w\in\nei(1)}
      \theta_{1w}Z_w\right\}}{1+\exp\left\{ 2\sum_{w\in\nei(1)}
      \theta_{1w}Z_w\right\}}; 
  \]
  recall~(\ref{eq:ising-logits}).  We obtain that
  \begin{align*}
    \var\left[ Z_1 \,|\,
      Z_{-1}\right]\;=\;
    4\var\left[ (Z_1+1)/2 \,|\,
      Z_{-1}\right]\;=\;\frac{4\exp\left\{ 2\sum_{w\in\nei(1)}
      \theta_{1w}Z_w\right\}}{\left(1+\exp\left\{ 2\sum_{w\in\nei(1)}
      \theta_{1w}Z_w\right\}\right)^2}.
  \end{align*}
  By assumption~\ref{ass:ising:2}, 
  \[
  -b_0\sqrt{q}\;\le\;\sum_{w\in\nei(1)}
  \theta_{1w}Z_w\;\le\;  b_0\sqrt{q}.
  \]
  It follows that
  \[
    \EE{(Z^\top u)^2} \;=\;\EE{\EE{(Z^\top u)^2 \,|\, Z_{-1}} } \;\ge\;
    \frac{4}{q}\cdot\frac{e^{ 2
        b_0\sqrt{q}}}{\left(1+e^{ 2b_0\sqrt{q}}\right)^2}. 
    \qedhere
  \]
\end{proof}

\section{Practical considerations when applying information criteria}
\label{sec:pract-cons}

Theorem~\ref{thm:cons-ebic-ising} shows that, with sufficient data,
application of $\mathrm{BIC}_\gamma$ allows one to identify the
correct set of edges, simultaneously at each node, with high
probability.  Application of the information criterion in practice,
however, faces two issues:
\begin{enumerate}[label=(\roman*)]
\item At an individual node, in order to find the sparse model that
  minimizes $\mathrm{BIC}_\gamma$, we must fit a large number of
  models.  With sparsity bounded by $q$, there are on the order of
  $p^q$ models, preventing an exhaustive search when the number of
  variables $p$ is large.
\item After performing neighborhood selection for each node, our
  results may be asymmetrical, that is, we might find that our estimates
  of the coefficients in~(\ref{eq:ising-logits}) satisfy
  $\hat{\beta}_{vw}\neq 0$ but $\hat{\beta}_{wv}=0$ for some pair of
  nodes $v,w$.
\end{enumerate}
To resolve the issue of the large number of possible models at each
node, it is common to use a computationally efficient procedure to
first produce a short list of candidate models, and then apply
$\mathrm{BIC}_\gamma$ to select from this list.  For each node, we use an
$\ell_1$-penalized logistic likelihood \citep{ravikumar:2010}  with varying
levels of penalization $\rho$ to produce the candidate models:
\begin{equation}\label{eqn:LogisticLasso}
\hat{\beta}^{(\rho)}_{v}=\arg\min_{\beta\in\R^{V\backslash\{v\}}}
\left\{-\sum_{i=1}^n \log\pr{\bigg(Z_{iv}\,\bigg|\, \sum_{w\neq v}Z_{iw}\beta_w\bigg)} +
  \rho\norm{\beta}_1\right\}
\end{equation}
where the probability term is given by the logistic model, i.e.
\[ \log\pr{\bigg(Z_{iv}\,\bigg|\, \sum_{w\neq v}Z_{iw}\beta_w\bigg)} = Z_{iv}\cdot  \sum_{w\neq v}Z_{iw}\beta_w - \mathrm{b}\bigg( \sum_{w\neq v}Z_{iw}\beta_w\bigg)\;.\]
As in Section~\ref{sec:cons-ebic-ising}, $Z_{iv}$ refers to the $v$-th
coordinate of the binary vector $Z_i=(Z_{i1},\dots,Z_{ip})$, which is
the $i$-th vector in a sample $Z_1,\dots,Z_n$.

To account for potential asymmetries when we compile information
across nodes, we follow the work of \citet{meinshausen:2006} and draw
an edge connecting nodes $v$ and $w$ based on either an \textsc{and}
rule (requiring both $\hat{\beta}_{vw}\neq 0$ and
$\hat{\beta}_{wv}\neq 0$) or an \textsc{or} rule (requiring only that either
$\hat{\beta}_{vw}\neq 0$ or $\hat{\beta}_{wv}\neq 0$); recall the
discussion from the introduction and, in particular, the empirical
study of \cite{hofling:2009}.

\section{Experiments}
\label{sec:experiments}

We study the performance of the extended BIC on both simulated and
real data.  The real data consists of precipitation measurements from
weather stations across the midwest, where we aim to recover a graph
that is consistent with the true geographical layout of the weather
stations. For this data set, we compare $\mathrm{BIC}_\gamma$ (with a
range of values for the parameter $\gamma$) with cross-validation as
well as with the stability selection method of
\cite{Meinshausen:2010}.  Our simulations replicate those of
\cite{ravikumar:2010}, including three different sparse graph
structures. For the simulated data, we compare different values of the $\gamma$
parameter for $\mathrm{BIC}_\gamma$.

\subsection{Simulated data}

\subsubsection{Data and methods for model selection}

We generate data from sparse Ising models associated to lattice graphs
and star graphs on $p$ nodes for $p\in\{64,100,225\}$.  For each graph
structure, the sample size $n$ is chosen based on the settings that
produced moderately high success rates in the simulations of
\cite{ravikumar:2010}.  We consider the following three graph types:
\begin{description}
\item[\rm\em $4$-Nearest neighbor lattice:] Arranging the nodes in a
  lattice of size $\sqrt{p}\times \sqrt{p}$, each node is connected to
  the nodes directly above, below, left or right, giving maximal
  degree $d=4$.  For adjacent nodes $v$ and $w$, we either set
  $\theta_{vw}=0.5$ (attractive couplings) or draw $\theta_{vw}$ at
  random from $\{+0.5,-0.5\}$ (random couplings).  The sample size is
  $n=\lceil 15d\log(p)\rceil$.
\item[\rm\em $8$-Nearest neighbor lattice:] Analogous to the above but
  also connecting nodes along diagonals.  The maximal degree is $d=8$.
  For edges $\{v,w\}$, we either set $\theta_{vw}=0.25$
  (attractive couplings), or draw $\theta_{vw}$ at random from
  $\{+0.25,-0.25\}$ (random couplings).  The sample size is $n=\lceil
  25d\log(p)\rceil$.
\item[\rm\em Star graph:] Edges are drawn from a designated ``hub''
  node to $q$ other nodes, where either $q=\lceil \log(p)\rceil$
  (logarithmic sparsity) or $q=\lceil 0.1 p\rceil$ (linear
  sparsity). For edges $\{v,w\}$, we set $\theta_{vw}=+0.25$.  The sample
  size is $n=\lceil 10d\log(p)\rceil$, where $d=q$ is
  the maximal degree of the graph.
\end{description}
These three graph structures are illustrated in Figure 1 of
\cite{ravikumar:2010}.

\begin{figure}[t]
\centering
(a) \includegraphics[width=10cm]{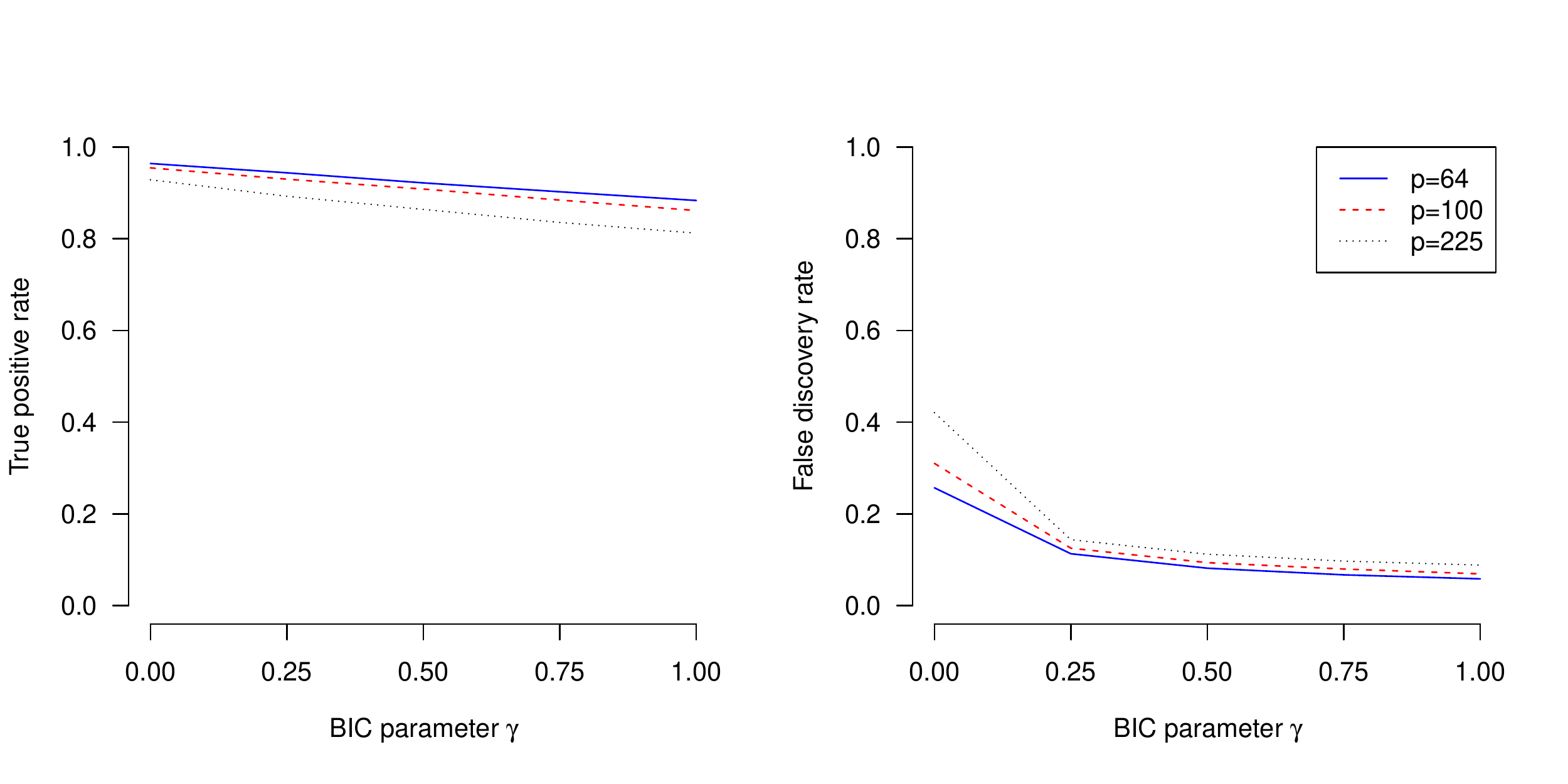}\\
(b) \includegraphics[width=10cm]{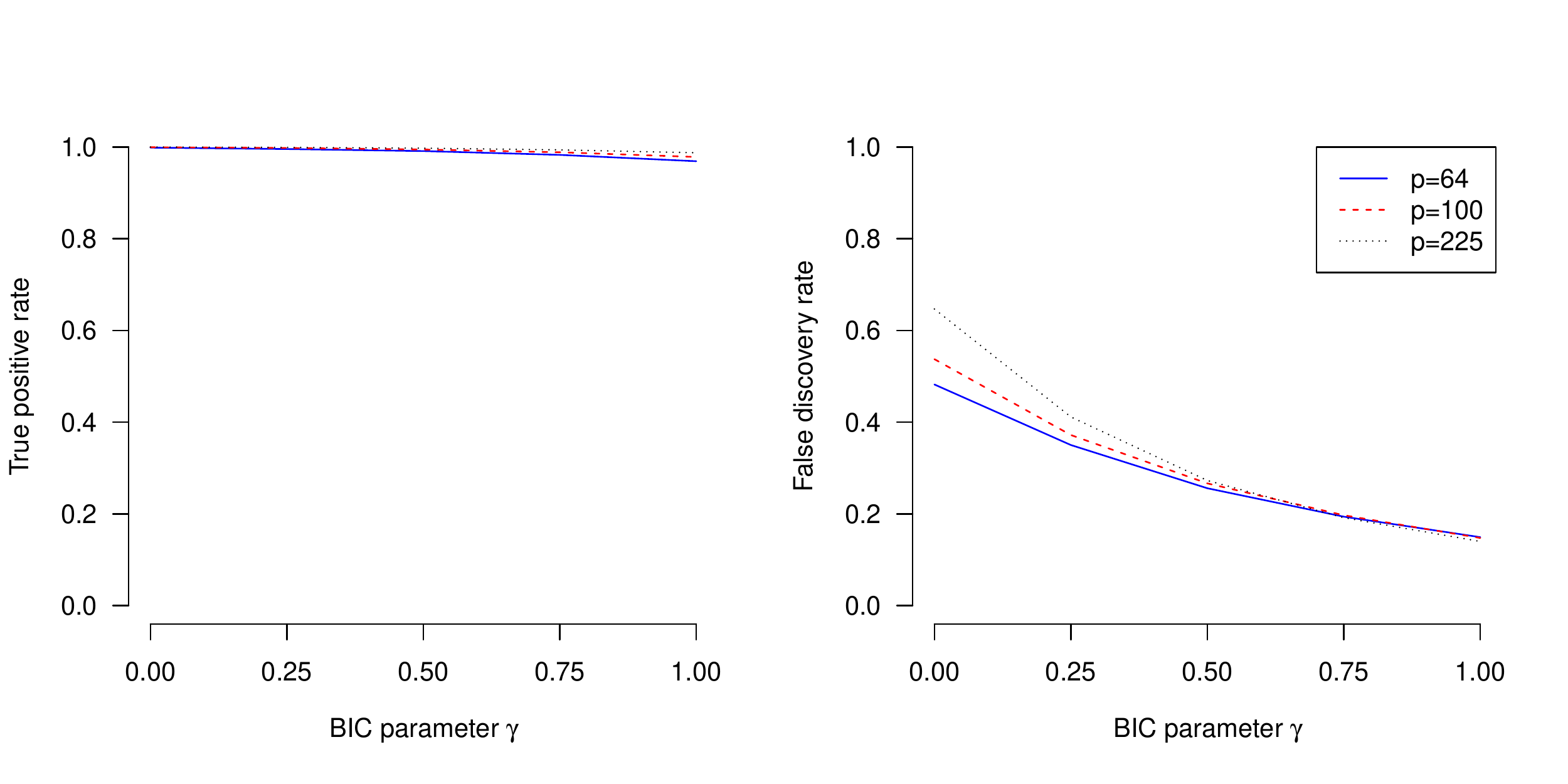}
\caption{Results for the $4$-nearest neighbor graph with (a)
  attractive couplings and (b) random couplings.} 
\label{fig:lattice4}
\end{figure}

\begin{figure}[t]
\begin{center}
(a) \includegraphics[width=10cm]{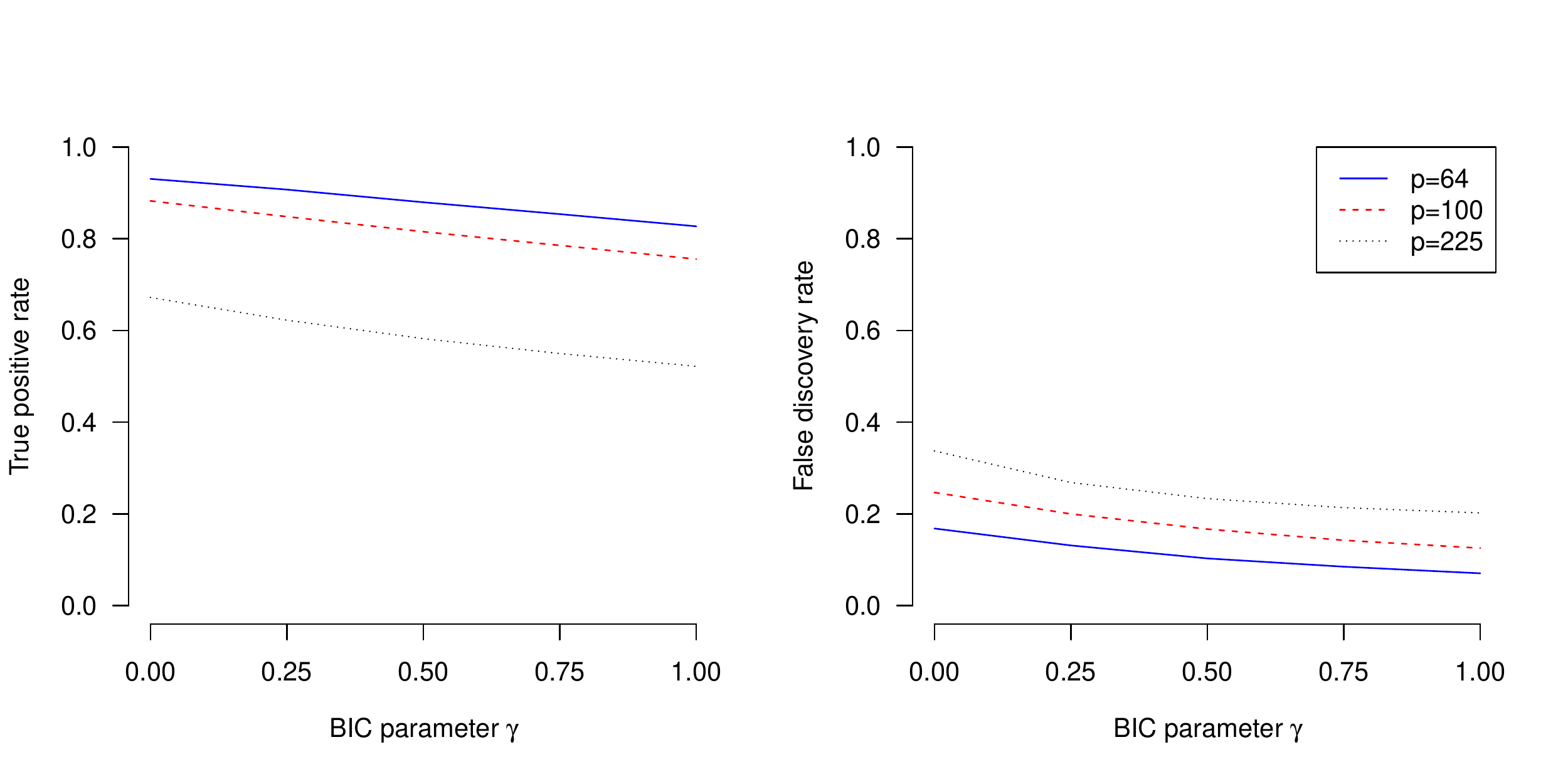}\\
(b) \includegraphics[width=10cm]{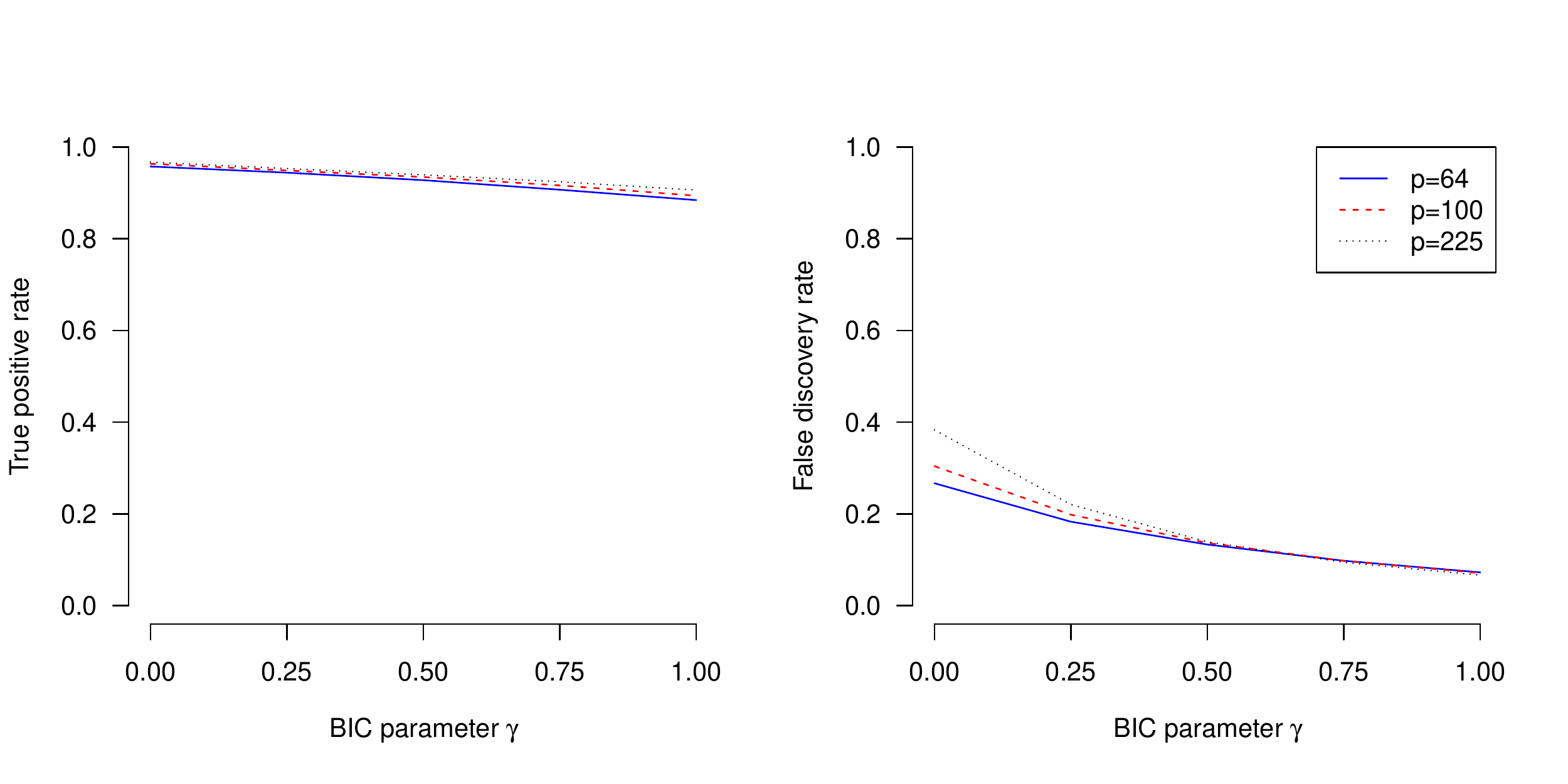}
\caption{Results for the $8$-nearest neighbor graph with (a)
  attractive couplings and (b) random couplings.} 
\label{fig:lattice8}
\end{center}
\end{figure}

\begin{figure}[t]
\begin{center}
(a) \includegraphics[width=10cm]{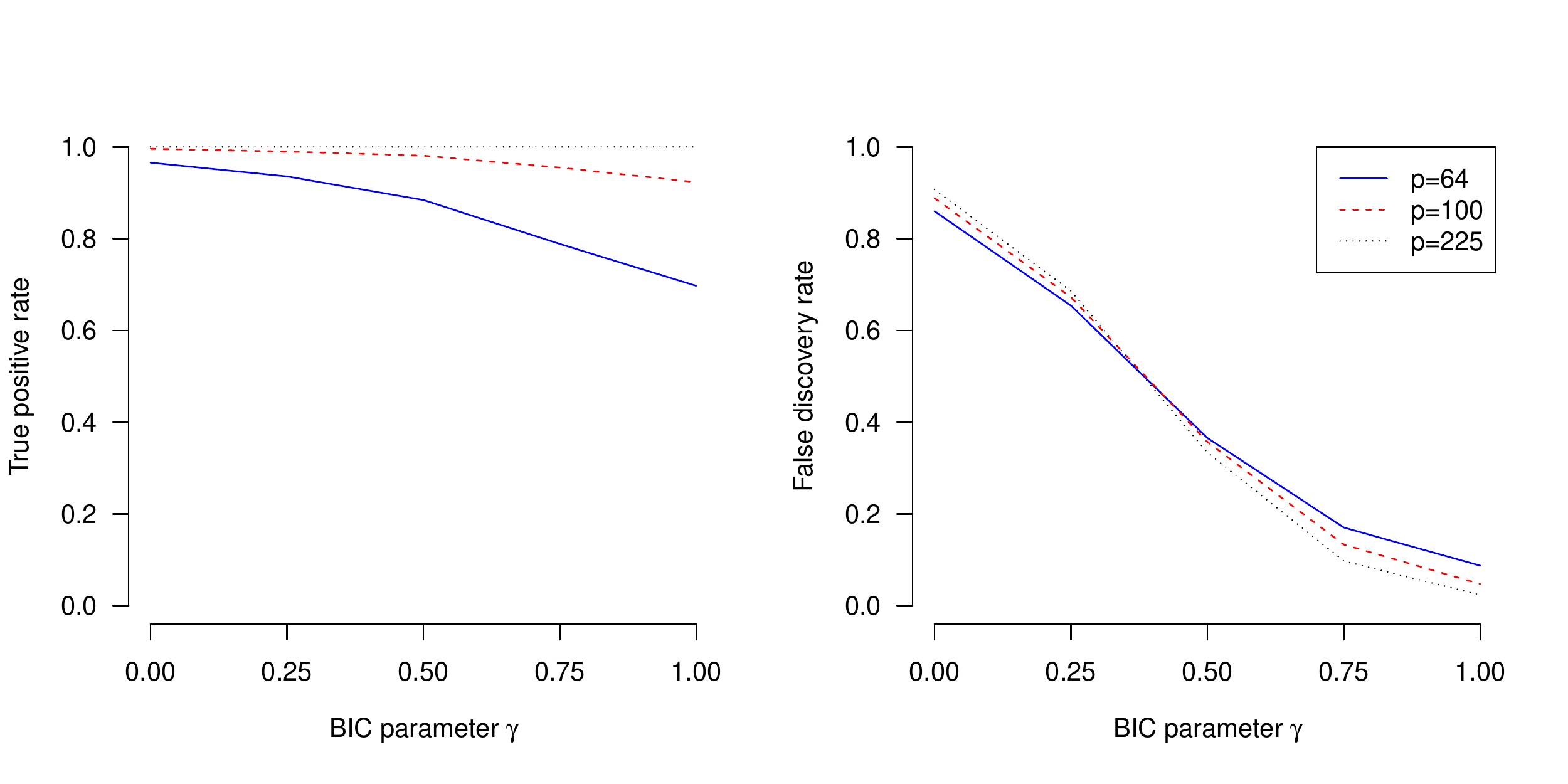}\\
(b) \includegraphics[width=10cm]{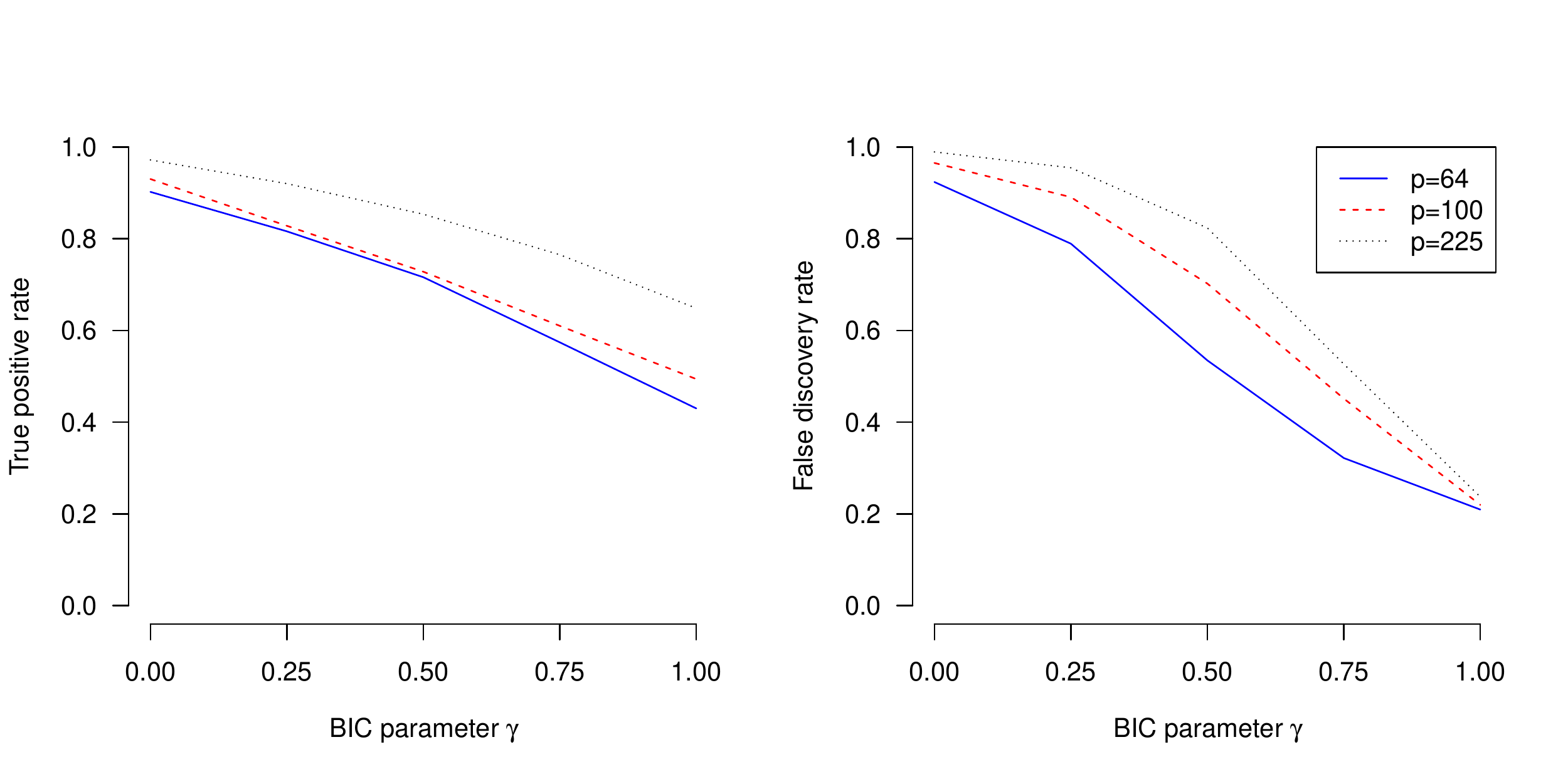}
\caption{Results for the star neighbor graph with (a) linear sparsity
  and (b) logarithmic sparsity.} 
\label{fig:star}
\end{center}
\end{figure}

For each of the three settings, we simulate 100 data sets.  Each time,
we perform nodewise $\ell_1$-penalized logistic regressions as
in~(\ref{eqn:LogisticLasso}), where we consider a wide range of
penalty parameters $\rho$ in order to produce a `path' of candidate
models for that node.  To this end, we used the \texttt{glmnet}
package for \texttt{R} \citep{Friedman:2010}.  For each node, we
then optimize $\mathrm{BIC}_\gamma$ in order to select a model from
the path.  Evaluating $\mathrm{BIC}_\gamma$ involves refitting each
candidate model without $\ell_1$-penalization, which was done using
the function \texttt{glm} in \texttt{R}.  We then symmetrized the
neighborhoods inferred by applying the \textsc{or} rule.  The
resulting graph is compared to the underlying true graph.  This
procedure was carried out for five choices for $\gamma$, namely,
$\gamma\in\{0,0.25,0.5,0.75,1\}$.

We note that the \textsc{and} rule for symmetrization led to
qualitatively similar conclusions, and we do not report the results
here.

\subsubsection{Results}
Results for the $4$- and $8$-nearest neighbor lattices as well as the
star graph are shown in
Figures~\ref{fig:lattice4},~\ref{fig:lattice8}, and~\ref{fig:star},
respectively.  For each scenario, we plot the positive selection rate
(proportion of true edges that are identified) and the false discovery
rate (the proportion of selected edges that are false positives). In
each case, we observe a tradeoff between positive selection rate and
false discovery rate as the parameter $\gamma$ for
$\mathrm{BIC}_\gamma$ varies.  Most notably, for nearly every setting
considered, we see that increasing $\gamma$ from $0$ to a positive
value can significantly reduce the false discovery rate without much
detriment to the positive selection rate, demonstrating a clear
benefit to using the extended BIC with $\gamma>0$ as opposed to the
ordinary $\mathrm{BIC}=\mathrm{BIC}_0$ for this high-dimensional
setting.

\subsection{Real data: regional weather patterns}

\subsubsection{Data and methods for model selection}

We apply $\mathrm{BIC}_\gamma$, and other competing methods, to the
task of inferring dependencies among binary indicators of
precipitation at $p=92$ weather stations across four states in the
Midwest region of the U.S.  The four states are Illinois, Indiana,
Iowa, and Missouri.  We fit models without taking the geographical
locations of the 92 stations into account, but then assess the
performance of different methods by referring to the distance between
weather stations.  Our rationale is that plausible graphs should
primarily link neighboring stations.  (One could argue that longer
links in East-West direction might be more reasonable than longer
North-South links but it seems difficult to quantify this and we did
not attempt to make such refined distinctions.)

The binary variables we consider indicate the existence of
precipitation at each station on a given day.  We model their joint
distribution with an Ising model as in~(\ref{eq:ising}) such that the
precipitation indicator at each node (weather station), conditional on
the observations from the other nodes, follows the logistic regression
model from~(\ref{eq:ising-logits}).  Following the same steps as in
our simulation study, we compute a set of candidate models for each
node using the $\ell_1$-penalized logistic regression and then select
a model from the set using either the ordinary
$\mathrm{BIC}=\mathrm{BIC}_0$ or $\mathrm{BIC}_\gamma$ with
$\gamma\in\{0.25,0.5\}$.  In addition, we considered cross-validation
and stability selection \citep{Meinshausen:2010}.  For
cross-validation, we select the model that minimizes average error on
test sets over $10$ folds.  For stability selection, we used the
\texttt{stabsel} function in the \texttt{mboost} package for
\texttt{R} \citep{R_mboost}, setting the expected support size to
$10$.\footnote{Parameters for the \texttt{stabsel} function were set at $\mathsf{q}=10$, the 
expected support size, and
$\mathsf{cutoff}=0.75$, the midpoint of the suggested range.}
As noted by \citet{Meinshausen:2010}, changing the settings
within a reasonable range did not have a large effect on the output.
For each of the mentioned methods, the node-wise edge selections are
compiled across all nodes to form a graph.  Performance is measured
relative to the true geographical layout of the weather stations,
which as mentioned above is ``unknown'' to the procedures we compare.

To give more specifics, we used data from the United States Historical
Climatology Network \citep{WeatherData}.\footnote{Available at
  \texttt{http://cdiac.ornl.gov/ftp/ushcn\_daily/}} The data consists
of weather-related variables that were recorded on a daily basis.  We
specifically gathered the precipitation data, which gives the total
amount of precipitation for each day.  Seasonality effects on
precipation are not as pronounced in the Midwest as in other parts of
the U.S., and we thus simply consider data from the entire
year.  However, to limit the effects of temporal dependencies between
successive observations, we took data from only the 1st and 16th day
of each month.  The resulting multivariate observations are then
treated as independent.  We removed weather stations where data
availability was low and discarded observations with missing values
for any of the remaining weather stations.  A total of $n=370$ days
and $p=92$ stations remained in the final data set.
Figure~\ref{fig:MidwestMap} shows a map of the 92 stations, along with
an undirected graph representing the Delaunay triangulation of the 92
locations.

\begin{figure}[t]
\begin{center}
\includegraphics[width=5cm]{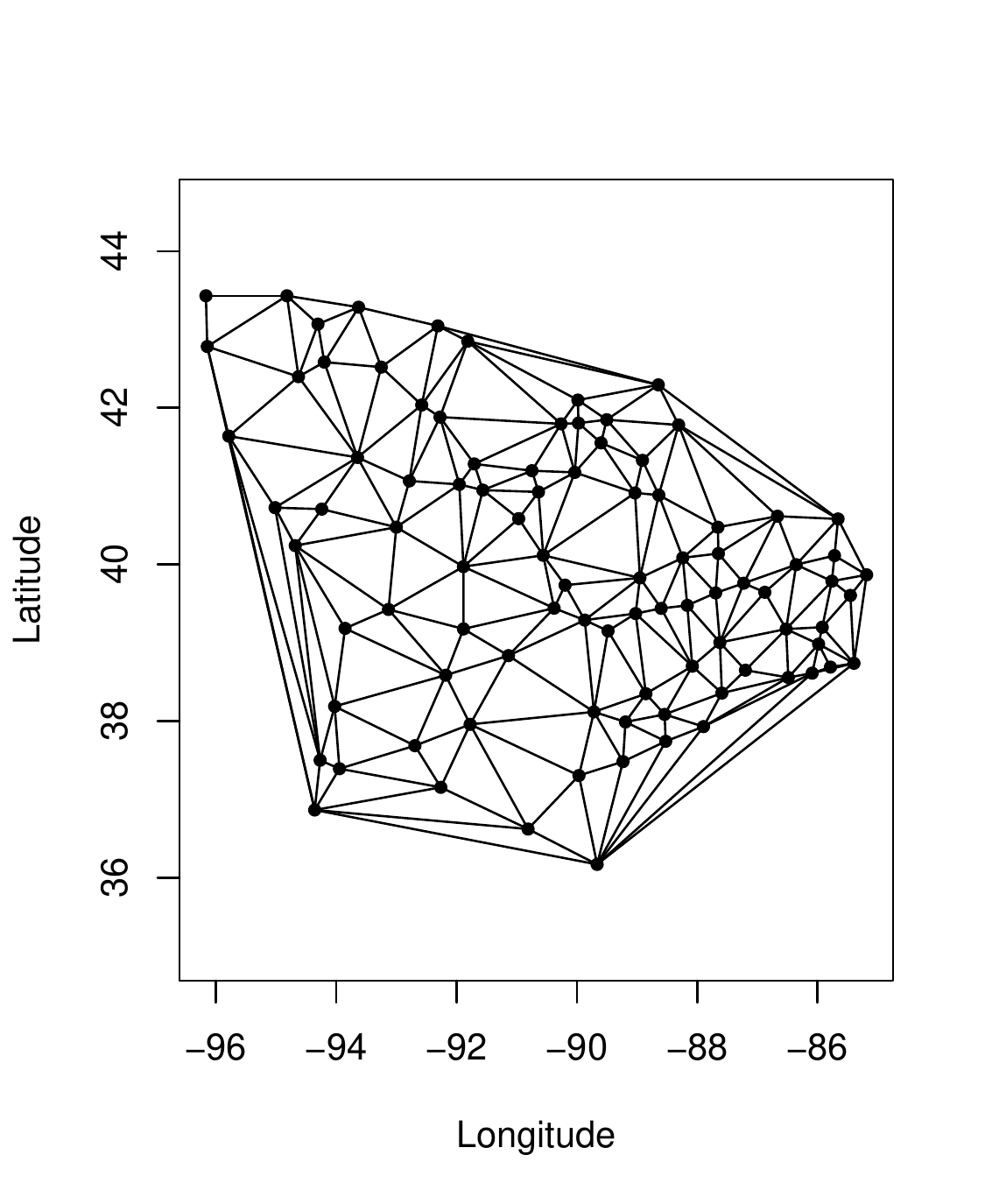}
\caption{Delaunay triangulation for 92 weather stations in Illinois,
  Indiana, Iowa, and Missouri.} 
\label{fig:MidwestMap}
\end{center}
\end{figure}

\subsubsection{Results}

To evaluate the model selection methods, we first compare the inferred
graphs to the geographic layout of the 92 stations by treating the
Delaunay triangulation as a ``true'' underlying graph for the
considered Ising model.  Table~\ref{table:WeatherData} shows the
results we obtain for each method, stated in terms of positive
selection rate (PSR) and false discovery rate (FDR), relative to the
``true'' Delaunay triangulation graph.  
Figure~\ref{fig:RecoveredGraphs} shows the recovered graphs under the
\textsc{and} and \textsc{or} combination rules.

        
\begin{table}[t]
        \caption{Positive selection rate  (\%) and false discovery rate (\%) in
          the weather data experiment,
           where the true graph is defined via the Delaunay
triangulation.} 
\begin{center}
\begin{tabular}{|@{\,}c@{\,}||r|r||r|r|}
\hline
&\multicolumn{2}{c||}{\textsc{and} rule}&\multicolumn{2}{c|}{\textsc{or} rule}\\
&PSR&FDR&PSR&FDR\\\hline\hline
BIC$_{0}$ & 41.98 & 32.93 & 55.73 & 46.72 \\\hline
BIC$_{0.25}$ & 37.40 & 27.94 & 52.67 & 42.02 \\\hline
BIC$_{0.5}$ & 34.73 & 26.61 & 50.38 & 38.89 \\\hline
Cross-validation & 59.16 & 57.65 & 71.37 & 75.65 \\\hline
Stability selection & 45.04 & 38.54 & 53.05 & 45.28 \\\hline
        \end{tabular}\end{center}
        \label{table:WeatherData}
        \end{table}


\begin{figure}[t]
\begin{center}
\includegraphics[width=12cm]{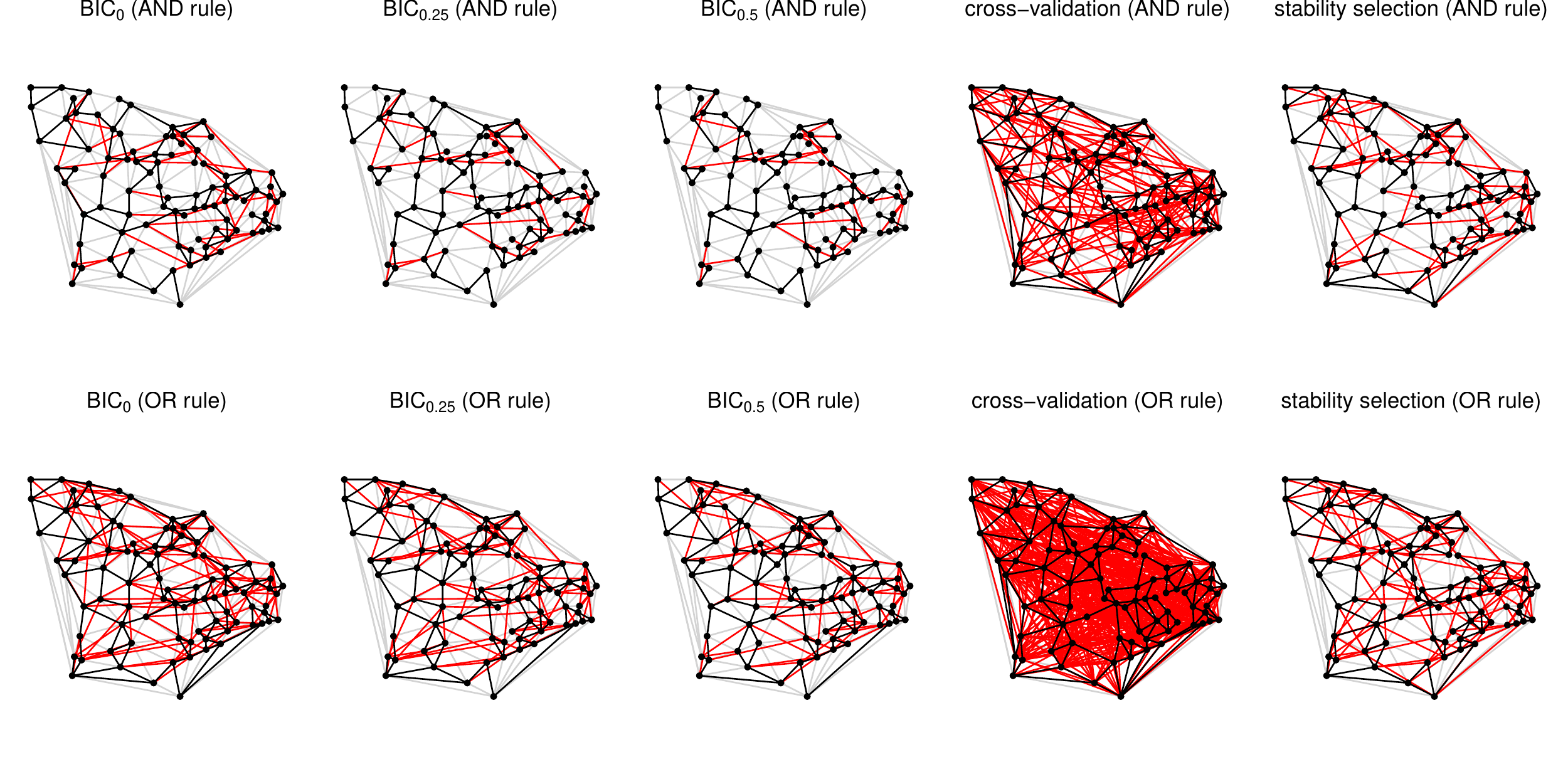}
\caption{Graphs recovered under each method. Black edges indicate true
  positives, red edges indicate false positives, and light gray edges
  indicate false negatives, i.e.\ true edges that were not recovered
  by the method, where the true graph is defined via the Delaunay
  triangulation.}
\label{fig:RecoveredGraphs}
\end{center}
\end{figure}

We see that cross-validation leads to a somewhat higher PSR than the
other methods, under either an \textsc{and} or an \textsc{or} rule.
However, this comes at the cost of a drastically higher FDR.  For
$\mathrm{BIC}_\gamma$, as we increase $\gamma$, we reduce the FDR at a
cost of a lower PSR, as expected.  Stability selection performs similarly 
to $\mathrm{BIC}_0$, but is computationally more expensive.


While it does not seem unreasonable to assume that the edges of the
Delaunay triangulation capture most of the strongest dependencies,
there might be additional dependencies that are not captured by the
edges in the triangulation.  For a different comparison of the methods
that more directly uses the geographic distances between the weather
stations, we apply Gaussian smoothing (scale: standard deviation = 10
miles) to estimate, as a function of $d$, the probability that a
method will infer an edge between two nodes that are $d$ miles apart.
The resulting functions are plotted in
Figure~\ref{fig:GraphSmoothedProbs}, which also includes the same
smoothed function calculation for the graph from the Delaunay
triangulation.

We observe that the smoothed function for the cross-validation methods
(under either the \textsc{or} or the \textsc{and} rule) does not decay
to zero as distance increases.  That is, in this experiment,
cross-validation selects a nonnegligible  proportion of edges
between nodes that are arbitrarily far apart, which is undesirable.
 To a lesser extent, the same problem occurs for stability selection
 combined with the \textsc{or} rule. 
The other methods, in contrast,
yield functions that do decay to zero relatively quickly
 as distance increases. Comparing
the methods that show the decay to zero, we see that for two
nearby weather stations, the $\mathrm{BIC}_\gamma$ methods combined
with the \textsc{or} rule are more likely to select an edge than any
of the remaining methods.  Overall, we
find that the information criteria perform well while requiring the
least amount of computation, and  increasing $\gamma$ provides a
useful trade-off between PSR and FDR.

\begin{figure}[t]
\begin{center}
\includegraphics[width=12cm]{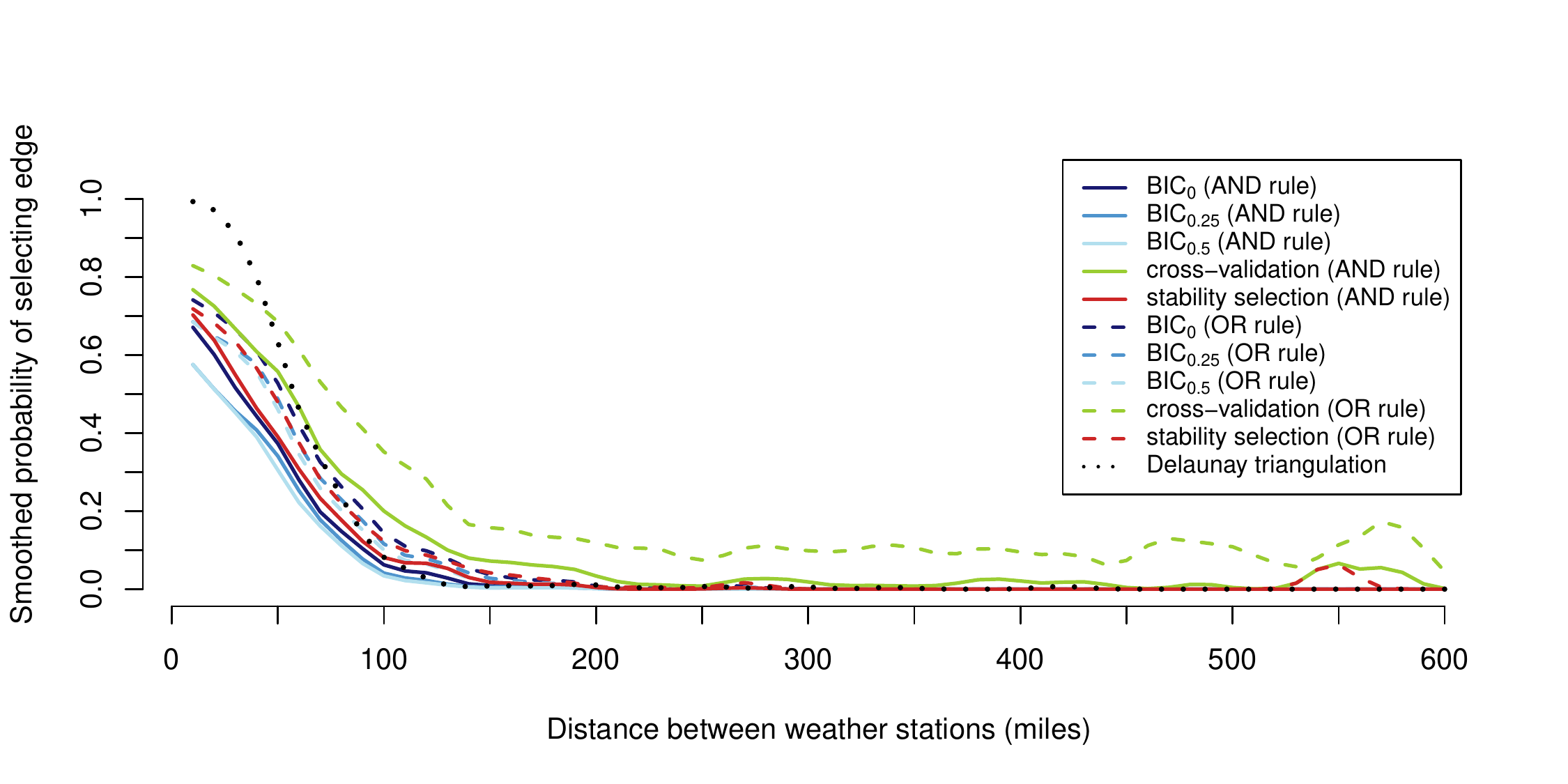}
\caption{Smoothed probability of selecting edges as a function of distance, for each method under the \textsc{or} rule and the \textsc{and} rule.}
\label{fig:GraphSmoothedProbs}
\end{center}
\end{figure}

\section{Discussion}
\label{sec:discussion}

As suggested by our numerical experiments and supported by our
theoretical analysis, Bayesian information criteria extended to
include a penalization term involving the number of covariates are
useful tools for variable selection in logistic regression as well as
neighborhood selection for Ising models.  The additional penalty term
can be motivated via a particular class of prior distributions on the
set of considered models.  We aim to discuss the formal connection
between fully Bayesian approaches and $\mathrm{BIC}_\gamma$ in a
subsequent paper; preliminary results under bounded sparsity are
described in the Ph.D.~thesis of the first author \citep{Foygel:2012}
and in a preprint \citep{Foygel:Drton:arxiv:2011}.

At the heart of this paper is an analysis of logistic regression with
random covariates.  While logistic regression has special properties,
our technical results can be extended to other generalized linear
models.  The main challenge for such generalizations is control of the
third derivative of the cumulant function which might no longer be
bounded.  Preliminary results under bounded sparsity can again be
found in \cite{Foygel:2012} and 
\cite{Foygel:Drton:arxiv:2011}.

\section*{Acknowledgments}

This work was supported by the U.S.~National Science Foundation under
Postdoctoral Fellowship No.~DMS-1203762 and Grant No.~DMS-1305154.  We
would like to thank Kean Ming Tan for helpful comments on draft
versions of this paper.

\appendix

\section{Why are second moments not sufficient?}
\label{sec:why-are-second}

Returning to the setup of Section~\ref{sec:logist-reg}, we recall that
our results on general logistic regression rely on
assumption~\ref{ass:A2}, which places an upper bound on third
moments.  In contrast, the lower bound in assumption~\ref{ass:A1}
concerns second moments (or, put differently, eigenvalues of small
submatrices of the covariance matrix).  It is tempting to try and
weaken our condition~\ref{ass:A2} to a sparse eigenvalue upper bound:
\begin{enumerate}[label=(A\arabic*$'$)]
  \addtocounter{enumi}{1}
  \item \label{ass:A2:second:moment}  
  For any $q$-sparse unit vector $u$, $\EE{(X_1^{\top}u)^2}\leq a_2'$.
\end{enumerate}

However, we now show that \ref{ass:A2:second:moment} is not sufficient
for the desired results in any asymptotic scenario where $q$ grows
with $n$, no matter how slow the growth is assumed to be. In
particular, we construct an example where, even though sparse
eigenvalues are bounded above and below, the Hessian conditions
assumed by \cite{luo:chen:2013:sii} do not hold at $\beta_0=0$ (i.e.\
$J_0=\emptyset$), recall \eqref{LuoChenBoundedChangeAssump}.

For simplicity, let $p=q$, and let $Z$ be a random vector that follows
a uniform distribution on $\{\pm1 \}^q$.  Let
$\mathbf{1}_q=(1,\dots,1)^\top$.  Then define a random vector $X$ by
setting
\[
X=\begin{cases}
  \mathbf{1}_q&\text{with prob.}\ \frac{1}{2q},\\
  -\mathbf{1}_q&\text{with prob.}\ \frac{1}{2q},\\
  Z&\text{with prob.}\ 1-\frac{1}{q}.
\end{cases}
\]
Clearly, $\EE{Z}=\EE{X}=0$, and $\EE{ZZ^{\top}}=\mathbf{I}_q$.
Therefore,
\[
\EE{XX^{\top}} = \frac{1}{q}\cdot \mathbf{1}_q\mathbf{1}_q^{\top} +
\left(1-\frac{1}{q}\right)\cdot \mathbf{I}_q,
\]
has minimal and maximal eigenvalue equal to
\begin{align*}
\lambda_{\min}(\EE{XX^{\top}}) &= 1-\frac{1}{q}, &
\lambda_{\max}(\EE{XX^{\top}}) &= 2 - \frac{1}{q}\;,
\end{align*}
respectively.  We observe that the eigenvalues are bounded above and
below by positive constants that are independent of $q$ (for all $q\geq 2$), as required by~\ref{ass:A1} and~\ref{ass:A2:second:moment}.

Now take the unit vector $u=\frac{1}{\sqrt{q}}\mathbf{1}_q$.  We see
that $|X^{\top}u|=\sqrt{q}$ with probability at least $1/q$.  For
independent random vectors $X_1,\dots,X_n$ that all have the same
distribution as $X$, it follows that $\#\{i:|X_i^\top u|=\sqrt{q}\}$
is at least as large as a
 $\mathrm{Binomial}(n,1/q)$ random variable.  Assume for
simplicity that $n/q$ is an integer.  Then $n/q$ is the median of the
$\mathrm{Binomial}(n,1/q)$ distribution, and so with probability at
least $\frac{1}{2}$, \[\#\{i:|X_i^\top u|=\sqrt{q}\}\geq
\frac{n}{q}\;.\]

In the remainder of this section, we prove that this property
contradicts the inequalities in \eqref{LuoChenBoundedChangeAssump},
which for $\beta_0=0$ state that
\[H(\beta)\preceq (1+\eps) H(0)\text{ for all }\norm{\beta}_2\leq \delta\;.\]
More precisely, for any $\eps>0$ there should be some $\delta=\delta(\eps)>0$
such that the statement holds, and the relationship between $\delta$ and $\eps$
(given by $\delta=\delta(\eps)$) should not depend on the dimensions of the problem.
  Note that since we have simplified the problem
by setting $p=q$, we do not need to make reference to submatrices of
$H(0)$.

Next take $\beta=u\cdot \frac{1}{\sqrt{q}}=\frac{1}{q}\mathbf{1}_q$; then
$\norm{\beta}_2=\frac{1}{\sqrt{q}}$.
 Since $\mathrm{b}''(0)\geq
\mathrm{b}''(z)$ and $\mathrm{b}''(z)=\mathrm{b}''(-z)$
for all $z\in\mathbb{R}$, we have
\begin{multline*}
u^{\top}\left(H(0)-H(\beta)\right)u 
\;=\;
\sum_{i=1}^n (X_i^{\top}u)^2 
\left(\mathrm{b}''(0)-\mathrm{b}''(X_i^{\top}\beta)\right) \\
 \geq \sum_{i: |X_i^{\top}u|=\sqrt{q}}(X_i^{\top}u)^2
\left(\mathrm{b}''(0)-\mathrm{b}''(X_i^{\top}\beta)\right)
= \sum_{i: |X_i^{\top}u|=\sqrt{q}} q\left(\mathrm{b}''(0)-\mathrm{b}''(1)\right)\\
\geq n\left(\mathrm{b}''(0)-\mathrm{b}''(1)\right)>0.05n\;,
\end{multline*}
where the next-to-last step holds with probability at least $\frac{1}{2}$ by the work above.

Now, in accordance with the conditions used by
\cite{luo:chen:2013:sii}, suppose that
\eqref{LuoChenBoundedChangeAssump} holds and that the Hessian is
bounded from above as $H(0)\preceq n\cdot c_2\mathbf{I}_q$, where $c_2$
is a constant that is independent of the dimensions $(n,q)$ of the problem.  Then for
the choice $\eps=0.05/c_2$, we require that there exists some
$\delta>0$, not depending on the dimensions $(n,q)$ of the problem,
such that 
\[
H(\beta)\;\succeq\; (1-\eps)H(0)\;\succeq\; H(0)-n\cdot \eps
c_2\mathbf{I}_q\;,
\] 
with high probability, for all $\beta\in\R^q$ with $\norm{\beta}_2\leq
\delta$. In particular, this implies that for the vector $u$ chosen
above, with high probability, for all $\beta\in\R^J$ with
$\norm{\beta}_2\leq \delta$,
\begin{equation}\label{eqn:counterexample_step}
0.05n = n\cdot \eps c_2 \;\geq\; u^{\top}\left(H(0)-H(\beta)\right)u\;.
\end{equation}
In particular, this must be true for $\beta=u\cdot \delta'$
for any $\delta'\leq \delta$. But from the work above, the bound
\eqref{eqn:counterexample_step} is not true for $\beta=u\cdot\frac{1}{\sqrt{q}}$,
and so we must have $\delta<\frac{1}{\sqrt{q}}$. This contradicts the requirement
that the relationship between $\eps$ and $\delta$ should not
depend on the dimensions of the problem.

\section{Proofs for likelihood and score results}\label{sec:LikelihoodAndScore}

This appendix is devoted to the proof of
Theorem~\ref{thm:HessianToLikelihood}, which gives bounds on
likelihood ratios for models postulating sparsity in the coefficient
vector $\beta$.  The bounds are for fixed values of the covariates
$X_1,\dots,X_n$ that satisfy the Hessian conditions from
Theorem~\ref{thm:RandomToHessian}.  All probability statements in this
section are tacitly understood to be conditional on $X_1,\dots,X_n$.


\subsection{Bounding the score function}

In this section, we prove bounds on the score function at the true
parameter $\beta_0$ that hold with high probability.  These bounds
concern the score function of true sparse models given by sets
$J\supseteq J_0$ with $|J|\le q$. 

Let $\eps'<\eps$ be a positive value that will be specified later.
For integer $r\ge 1$, let $\tau_r,\tilde{\tau}_r>0$ be defined via
\[
\tau_{r}^2\coloneqq
{\frac{2}{(1-\eps')^3}}\cdot\left[{(|J_0|+r)\log\left(\frac{3}{\eps'}\right)
  + \log(4p^{\nu})+ r\log(2p)}\right]
\] 
and
\[
\tilde{\tau}_{r}^2\coloneqq
{\frac{2}{(1-\eps')^3}}\cdot\left[{r\log\left(\frac{3}{\eps'}\right)
  + \log(4p^{\nu})+ r\log(2p)}\right]\;,
\]
respectively.  Assume that
\begin{equation}
  \label{eq:tau:r}
  \tau_{r}\leq
  \frac{\eps'\sqrt{n\,c_{\mathrm{lower}}(\|\beta_0\|_2)^3}
  }{(1-\eps')c_{\mathrm{change}}} 
\end{equation}
for $r\le q-|J_0|$.  This assumption can be guaranteed to
hold by choosing $C_{\mathrm{sample},1}$ in the statement of
Theorem~\ref{thm:HessianToLikelihood} appropriately.  

\begin{lemma}
  \label{lem:score:bound}
  Fix values for the observations $X_1,\dots,X_n$ that satisfy the
  Hessian conditions~(\ref{Hess1}) and~(\ref{Hess2}) from
  Theorem~\ref{thm:RandomToHessian}.  Assume further that the
  inequality in~(\ref{eq:tau:r}) holds.  Then with conditional
  probability at least $1-p^{-\nu}$, we have for all $J\supseteq J_0$
  with $|J|\leq q$ that both
  \begin{equation}
    \label{ScoreBoundEq}
    \Norm{H_J(\beta_0)^{-\frac{1}{2}}s_J(\beta_0)}_2\leq
    \tau_{|J\backslash J_0|}
  \end{equation}
  and
  \begin{equation}
    \label{ScoreBoundEq2}
    \Norm{\mathrm{Proj}_{\mathcal{S}_J^{\perp}}
      \left(H_J(\beta_0)^{-\frac{1}{2}}s_J(\beta_0)\right)}_2\leq 
    \tilde{\tau}_{|J\backslash J_0|}\;,
  \end{equation}
  where the projection is onto the orthogonal complement of the subspace
  \[
  \mathcal{S}_J=\left\{H_J(\beta_0)^{\frac{1}{2}}z\::\:z\in\R^{J_0}\right\}\subset
  \R^J\;.
  \]
\end{lemma}
To be clear, in the definition of $\mathcal{S}_J$, we use $\R^{J_0}$
to denote the coordinate subspace of vectors $z\in\R^J$ with $z_j=0$
for all $j\in J\setminus J_0$.
\begin{proof}
  We will establish the bounds in~(\ref{ScoreBoundEq})
  and~(\ref{ScoreBoundEq2}) by using an $\epsilon$-net argument based
  on the fact that for any vector $z\in\mathbb{R}^p$,
  \[
  \|z\|_2=\sup\left\{\, u^\top z \,:\, u\in\mathbb{R}^p,\, \|u\|_2=1\,\right\}.
  \]
  To prepare for the argument, fix a superset $J\supseteq J_0$, a vector
  $u\in\R^J$, and a scalar $\tau>0$.  Observe that
  \begin{multline}
    \label{eq:score:bound:ineq}
    \pr\left\{u^{\top}H_J(\beta_0)^{-\frac{1}{2}}s_J(\beta_0)>\tau
      \Big\vert X\right\} \\
    \leq \EE{\exp\left\{\tau\cdot
        u^{\top}H_J(\beta_0)^{-\frac{1}{2}}s_J(\beta_0)-
        \tau^2\right\}
      \Big\vert  X}.
  \end{multline}
  By definition, 
  \begin{equation}
    \label{eq:scoreJ}
    s_J(\beta_0)=\sum_{i=1}^n
    X_{iJ}(Y_i-\bb'(X_i^{\top}\beta_0)),
  \end{equation}
  and since the conditional distribution of $Y_i$ given $X_i$ belongs to
  an exponential family, we have 
  \begin{equation}
    \label{eq:mgf:y}
    \EE{\exp\{s Y_i\}\vert
      X_i}=\exp\left\{\bb(X_i^{\top}\beta_0+s)-\bb(X_i^{\top}\beta_0)\right\}.
  \end{equation}
  Plugging~(\ref{eq:scoreJ}) into~(\ref{eq:score:bound:ineq}) and
  using~(\ref{eq:mgf:y}), we obtain that
  \begin{align*}
    &\log
    \pr\left\{u^{\top}H_J(\beta_0)^{-\frac{1}{2}}s_J(\beta_0)>\tau
      \Big\vert X\right\} \\
    &\le \sum_{i=1}^n \left[\bb\left(X_i^{\top}(\beta_0+\tau 
        H_J(\beta_0)^{-\frac{1}{2}}u)\right)
      -\bb\left(X_i^{\top}\beta_0\right)\right]\\
    &\qquad\qquad\qquad - \sum_{i=1}^n \left[
      \bb'(X_i^{\top}\beta_0)\cdot \tau 
      X_{iJ}^{\top}H_J(\beta_0)^{-\frac{1}{2}}u\right]  \;-\;
    \tau^2
    \\
    &=\frac{1}{2}\sum_{i=1}^n\left[
      \bb''\left(X_i^{\top}(\beta_0+\xi\cdot \tau
        H_J(\beta_0)^{-\frac{1}{2}}u)\right)\cdot
      \left(\tau X_{iJ}^{\top}H_J(\beta_0)^{-\frac{1}{2}}u\right)^2\right]
    \;-\; \tau^2, 
  \end{align*}
  where the last equation is a 2nd-order Taylor expansion with
  $\xi\in[0,1]$.  We may rewrite the inequality just obtained as
  \begin{multline*}
    \log
    \pr\left\{u^{\top}H_J(\beta_0)^{-\frac{1}{2}}s_J(\beta_0)>\tau\Big\vert
      X\right\}
    \leq \\
    \frac{\tau^2}{2} u^{\top}H_J(\beta_0)^{-\frac{1}{2}}
    H_J\left(\beta_0+\xi\cdot \tau
      H_J(\beta_0)^{-\frac{1}{2}}u\right)H_J(\beta_0)^{-\frac{1}{2}}u
    \;-\; \tau^2\;.
  \end{multline*}
  Now, for
  $\tau=\tau'_r\coloneqq\tau_{r}(1-\eps')$ with $r=|J\setminus
  J_0|$ and a vector
  $u\in\mathbb{R}^J$ with $\norm{u}_2\leq 1$, it holds that
  \[
  \Norm{\xi\cdot \tau'_r
    H_J(\beta_0)^{-\frac{1}{2}}u}_2\leq
  \tau_{r}(1-\eps')\cdot\sqrt{\frac{1}{n\, c_{\mathrm{lower}}(\|\beta_0\|_2) }}\leq 
  \eps'\cdot \frac{c_{\mathrm{lower}}(\|\beta_0\|_2)}{c_{\mathrm{change}}}\;;
  \]
  recall~(\ref{eq:tau:r}).  Via~(\ref{LuoChenBoundedChangeAssump}) and~(\ref{eq:luo:chen:delta}), the assumed
  Hessian conditions imply that
  \begin{multline*}
    H_J(\beta_0)^{-\frac{1}{2}}H_J\left(\beta_0+\xi\cdot \tau'_{r}\cdot
      H_J(\beta_0)^{-\frac{1}{2}}u\right)H_J(\beta_0)^{-\frac{1}{2}} \\
    \preceq\;
    H_J(\beta_0)^{-\frac{1}{2}}\big[(1+\eps')H_J(\beta_0)\big]
    H_J(\beta_0)^{-\frac{1}{2}}\;=\;
    (1+\eps')\cdot \mathbf{I}_J\;,
  \end{multline*}
  and thus
  \begin{multline}\label{ExpoProb}
    \pr\left\{u^{\top}H_J(\beta_0)^{-\frac{1}{2}}s_J(\beta_0)>\tau'_{r}
      \Big\vert X\right\} \\
    \leq\; \exp\left\{\frac{\tau'_{r}{}^2}{2} (1+\eps')- \tau'_{r}{}^2\right\}
    =\exp\left\{-\frac{\tau'_{r}{}^2}{2}\left(1 -\eps'\right)\right\}\;.
  \end{multline}
  
  Next, let $\mathcal{U}_J$ be an $\eps'$-net for the unit sphere in
  $\R^J$ with respect to the Euclidean norm, that is, $\mathcal{U}_J$
  is a subset of the sphere such that for any unit vector $v$ there
  exists a (unit) vector $u\in\mathcal{U}_J$ such that
  $\|u-v\|_2<\epsilon'$.  In particular, for the unit vector
  \[
  v=\frac{H_J(\beta_0)^{-\frac{1}{2}}s_J(\beta_0)}{
    \Norm{H_J(\beta_0)^{-\frac{1}{2}}s_J(\beta_0)}_2} 
  \]
  and corresponding $u\in\mathcal{U}_J$ with $\norm{u-v}_2\leq \eps'$,
  we see that
  \[u^{\top}v = v^{\top}v + (u-v)^{\top}v \geq \norm{v}^2_2 - \norm{u-v}_2 \cdot\norm{v}_2 \geq 1-\eps'\;,\]
  and so
  \begin{equation}
    \label{eq:net:to:norm}
    u^{\top}H_J(\beta_0)^{-\frac{1}{2}}s_J(\beta_0)\geq 
    (1-\eps')\Norm{H_J(\beta_0)^{-\frac{1}{2}}s_J(\beta_0)}_2\;.
  \end{equation}

  We can take the $\epsilon$-net such that
  \begin{equation}
    \label{eq:net:size}
    \left|\mathcal{U}_J\right|\leq
    \left(1+\frac{2}{\epsilon'}\right)^{|J|}\leq 
    \left(\frac{3}{\eps'}\right)^{|J|}\; ;
  \end{equation}
  see Proposition 1.3 in Chapter 15 of \cite{lorentz:etal:1996} or Lemma
  14.27 in \cite{buhlmann:vandegeer:2011}.
  Inequality~(\ref{eq:net:to:norm}) and a union bound yield that
  \begin{align*}
    \pr&\left\{\Norm{H_J(\beta_0)^{-\frac{1}{2}}s_J(\beta_0)}_2>
      \tau_{r}\right\}\\
    &\leq\pr\left\{u^{\top}H_J(\beta_0)^{-\frac{1}{2}}s_J(\beta_0)\geq
      \tau'_{r}\text{ for some }u\in\mathcal{U}_J\right\}\\ 
    &\leq\left|\mathcal{U}_J
    \right|\cdot\pr\left\{u^{\top}H_J(\beta_0)^{-\frac{1}{2}}s_J(\beta_0)\geq 
      \tau'_{r}\text{ for any single }u\in\mathcal{U}_J\right\}.
  \end{align*}
  Applying inequalities~(\ref{ExpoProb}) and~(\ref{eq:net:size}), and
  plugging in the definition of $\tau'_{r}$, we obtain that
  \begin{multline}
    \label{eq:net:argument:ineq:one:point}
    \pr\left\{\Norm{H_J(\beta_0)^{-\frac{1}{2}}s_J(\beta_0)}_2>
      \tau_{r}\right\} \;\leq\; \left(\frac{3}{\eps'}\right)^{|J|}\cdot
    \exp\left\{-\frac{\tau'_{r}{}^2}{2}\left(1 -\eps'\right)\right\} \\
    \;=\; \exp\left\{-\log(4p^{\nu})- r\log(2p)\right\} 
    \;=\;
    \frac{1}{4(2p)^{r}}\cdot\frac{1}{p^{\nu}}\;. 
  \end{multline}
  
  Finally, to consider all sets $J\supseteq J_0$ with $|J|\leq q$
  simultaneously, we apply the union bound
  \begin{align*}
    &\pr\left\{\Norm{H_J(\beta_0)^{-\frac{1}{2}}s_J(\beta_0)}_2>
      \tau_{|J\backslash J_0|}\text{ for some }J\supseteq J_0,|J|\leq
      q\right\}\\ 
    &\leq \sum_{r=0}^{q-|J_0|}
    \pr\left\{\Norm{H_J(\beta_0)^{-\frac{1}{2}}s_J(\beta_0)}_2\geq
      \tau_{r}\text{ for some }J\supseteq J_0\text{ with } |J\backslash
      J_0|=r\right\}.
  \end{align*}
  Using the fact that there are at most ${p\choose r}\leq p^{r}$ sets
  $J\supseteq J_0$ with $|J\backslash J_0|=r$,
  inequality~(\ref{eq:net:argument:ineq:one:point}) and another union
  bound imply that
  \begin{multline*}
    \pr\left\{\Norm{H_J(\beta_0)^{-\frac{1}{2}}s_J(\beta_0)}_2>
      \tau_{|J\backslash J_0|}\text{ for some }J\supseteq J_0,|J|\leq
      q\right\}\\
    \;\leq\;
    \sum_{r=0}^{q-|J_0|}
    p^r\cdot\frac{1}{4(2p)^{r}}\cdot\frac{1}{p^{\nu}} 
    \;\leq\;
    \frac{1}{4p^{\nu}} \sum_{r=0}^{\infty} \frac{1}{2^r}
    \;=\;
    \frac{1}{2p^{\nu}}\;.
  \end{multline*}

  To prove the analogous statement about the projection operator, we
  instead take $\mathcal{U}_J$ to be an $\eps'$-net of the unit sphere
  in the orthogonal complement $\mathcal{S}_J^{\perp}\subset \R^J$,
  which has dimension $|J\backslash J_0|$.  Consequently, we have
  $|\mathcal{U}_J|\leq (3/\eps')^{|J\backslash J_0|}$.  The rest of
  the argument proceeds identically with a bound of $1/(2p^\nu)$ for
  the probability of the considered event.  A union bound over the two
  cases gives the claimed bound of $1/p^\nu$ for the probability of
  both inequalities holding.
\end{proof}

\subsection{Bounding the likelihood function}

In this subsection we analyze the log-likelihood ratios of sparse
models given by sets $|J|\leq q$, proving Theorem
\ref{thm:HessianToLikelihood}.  It suffices to show that the two
statements \ref{thm2:a} and \ref{thm2:b} in
Theorem~\ref{thm:HessianToLikelihood} are implied by the bounds
\eqref{ScoreBoundEq} and \eqref{ScoreBoundEq2} from
Lemma~\ref{lem:score:bound}.  The probability of the latter bounds
holding was shown to be large in the previous subsection.  In our
proof we consider a fixed vector $\beta_0$.  The statement being true
uniformly for vectors with $\norm{\beta_0}_2$ bounded by $a_0$ follows
from the monotonicity of the functions $c_\mathrm{lower}$ and
$c_\mathrm{upper}$.

Fix any $J\supseteq J_0$ with $|J|\leq q$.  Consider any $\beta\in
\R^J$ and let $\gamma=\beta-\beta_0$.  
Let
\[
\tilde{\gamma}=
H_J(\beta_0)^{-\frac{1}{2}}\cdot\mathrm{Proj}_{\mathcal{S}_J}\left(H_J(\beta_0)^{\frac{1}{2}}\gamma\right)\in\R^{J_0}\;,
\]
where $\mathcal{S}_J\subset \R^J$ is the $|J_0|$-dimensional subspace
defined in Lemma~\ref{lem:score:bound}.  By definition,
$H_J(\beta_0)^{\frac{1}{2}}\tilde{\gamma}=
\mathrm{Proj}_{\mathcal{S}_J}\left(H_J(\beta_0)^{\frac{1}{2}}\gamma\right)$, 
and thus
\begin{equation}
  \label{eq:likelihoood:bound:pythagoras}
  \Norm{H_J(\beta_0)^{\frac{1}{2}}\gamma}^2_2 =
  \Norm{H_J(\beta_0)^{\frac{1}{2}}\tilde{\gamma}}^2_2 +
  \Norm{H_J(\beta_0)^{\frac{1}{2}}(\gamma - \tilde{\gamma})}^2_2\;.
\end{equation}
Using~(\ref{Hess1}), we obtain that
\begin{align}
  \nonumber
\norm{\tilde{\gamma}}_2
&\leq
\frac{1}{\sqrt{nc_{\mathrm{lower}}(\|\beta_0\|_2)}}
\Norm{\mathrm{Proj}_{\mathcal{S}_J} 
\left(H_J(\beta_0)^{\frac{1}{2}}\gamma\right)}_2\\
  \nonumber
&\leq 
\frac{1}{\sqrt{nc_{\mathrm{lower}}(\|\beta_0\|_2)}}
\Norm{H_J(\beta_0)^{\frac{1}{2}}\gamma}_2\\
\label{eq:likelihood:bound:ineq:tilde}
&\leq
\sqrt{\frac{ c_{\mathrm{upper}}(\|\beta_0\|_2) }{
    c_{\mathrm{lower}}(\|\beta_0\|_2) }}\;\norm{\gamma}_2\;.
\end{align}

We now compare the values of the log-likelihood function at $\beta_0$,
$\beta_0+\gamma$, and $\beta_0+\tilde{\gamma}$, using
Taylor-expansions.  Using
Proposition~\ref{prop:hess:in:luo:chen:form}, we calculate
\begin{multline}
  \label{eqn:LikelihoodBoundOriginal}
  \ln(\beta_0+ \gamma)-\ln(\beta_0)
  =s_J(\beta_0)^{\top} \gamma- \frac{1}{2} \gamma ^{\top}H_J(\beta_0 +
  \xi\cdot \gamma)\gamma\\ 
  \leq s_J(\beta_0)^{\top} \gamma - \frac{1}{2}\left(1 -
    \frac{c_{\mathrm{change}}}{c_{\mathrm{lower}}(\|\beta_0\|_2)}\cdot
    \norm{\gamma}_2\right) \;\gamma ^{\top}H_J(\beta_0) \gamma
\end{multline}
and
\begin{multline}
  \label{eqn:LikelihoodBoundProj}
  \ln(\beta_0+\tilde{\gamma})-\ln(\beta_0)
  =s_J(\beta_0)^{\top}\tilde{\gamma} - \frac{1}{2}\tilde{\gamma}^{\top}H_J(\beta_0 + \tilde\xi\cdot\tilde{\gamma})\tilde{\gamma}\\
  \geq s_J(\beta_0)^{\top}\tilde{\gamma} - \frac{1}{2}\left(1 +
    \frac{c_{\mathrm{change}}}{c_{\mathrm{lower}}(\|\beta_0\|_2)}\cdot
    \norm{\tilde{\gamma}}_2\right)\;
  \tilde{\gamma}^{\top}H_J(\beta_0)\tilde{\gamma}\;,
\end{multline}
where $\xi,\tilde\xi\in [0,1]$.
Subtracting~(\ref{eqn:LikelihoodBoundProj})
from~(\ref{eqn:LikelihoodBoundOriginal}) and
using~(\ref{eq:likelihoood:bound:pythagoras}), we find that
\begin{multline*}
\ln(\beta_0+\gamma)-\ln(\beta_0+\tilde{\gamma})
\;\leq\;
s_J(\beta_0)^{\top}(\gamma-\tilde{\gamma})
\\- \frac{1}{2}\Norm{H_J(\beta_0)^{\frac{1}{2}}(\gamma - \tilde{\gamma})}^2_2
 + \frac{c_{\mathrm{change}}\left(
\norm{\gamma}_2 \gamma ^{\top}H_J(\beta_0) \gamma  
 + \norm{\tilde{\gamma}}_2
 \tilde{\gamma}^{\top}H_J(\beta_0)\tilde{\gamma}\right)}{2c_{\mathrm{lower}}(\|\beta_0\|_2)}\; .
\end{multline*}
Inequalities~(\ref{Hess1}) and~(\ref{eq:likelihood:bound:ineq:tilde})
yield that 
\begin{multline}
  \label{eqn:LikehoodBoundParts}
  \ln(\beta_0+\gamma)-\ln(\beta_0+\tilde{\gamma})\;\leq\;
  s_J(\beta_0)^{\top}(\gamma-\tilde{\gamma})\\
  - \frac{1}{2}\Norm{H_J(\beta_0)^{\frac{1}{2}}(\gamma - \tilde{\gamma})}^2_2
  + n\cdot c_{\mathrm{change}}\left(\frac{
      c_{\mathrm{upper}}(\|\beta_0\|_2)}{c_{\mathrm{lower}}(\|\beta_0\|_2)}\right)^{\frac{3}{2}}\norm{\gamma}_2^3
  \;.
\end{multline}
Writing
\[
s_J(\beta_0)^{\top}(\gamma-\tilde{\gamma}) \;=\; \left(H_J(\beta_0)^{-\frac{1}{2}}s_J(\beta_0)\right)^{\top}\left(H_J(\beta_0)^{\frac{1}{2}}(\gamma-\tilde{\gamma})\right) 
\]
and noting that $H_J(\beta_0)^{\frac{1}{2}}(\gamma - \tilde{\gamma})
\in \mathcal{S}_J^{\perp}$, we see that the first two terms of the
bound in \eqref{eqn:LikehoodBoundParts} can be bounded as
\begin{align*}
  s_J(\beta_0)^{\top}(\gamma-\tilde{\gamma})
  - \frac{1}{2}\Norm{H_J(\beta_0)^{\frac{1}{2}}(\gamma -
    \tilde{\gamma})}^2_2
  &\leq \sup_{z\in\mathcal{S}_J^{\perp}} \left(H_J(\beta_0)^{-\frac{1}{2}}s_J(\beta_0)\right)^{\top}z
  - \frac{1 }{2}\norm{z}^2_2\\
  &= \frac{1}{2}\Norm{\mathrm{Proj}_{\mathcal{S}_J^{\perp}}\left(H_J(\beta_0)^{-\frac{1}{2}}s_J(\beta_0)\right)}^2_2\;,
\end{align*}
which is at most $\tilde{\tau}_{|J\backslash J_0|}^2/2$ by the
assumed inequality~(\ref{ScoreBoundEq2}).

Consider now the MLE $\beta=\wh{\beta}_J=\beta_0+\gamma$, and define
$\tilde{\gamma}\in\R^{J_0}$ as before.   Then $\ln(\wh{\beta}_{J_0})\geq\ln(\beta_0+\tilde{\gamma})$
because
$\beta_0+\tilde{\gamma}\in\R^{J_0}$,
and so applying the calculations above, we have
\begin{align}
  \nonumber
  \ln(\wh{\beta}_J)-\ln(\wh{\beta}_{J_0})
  &\leq \ln(\wh{\beta}_J)-\ln(\beta_0+\tilde{\gamma})\\
  \label{eq:score:bound:diff:maxima}
  &\leq \frac{1}{2}\tilde{\tau}_{|J\backslash J_0|}^2+ n\cdot
  c_{\mathrm{change}}\left(\frac{c_{\mathrm{upper}}(\|\beta_0\|_2)}{
      c_{\mathrm{lower}}(\|\beta_0\|_2)}\right)^{\frac{3}{2}} 
  \cdot \norm{\wh{\beta}_J-\beta_0}^3_2\;.
\end{align}
We can thus bound the difference between the maxima of the
log-likelihood functions if we can bound the distance
$\norm{\wh{\beta}_J-\beta_0}_2$.

To bound $\norm{\wh{\beta}_J-\beta_0}_2$, we return to
\eqref{eqn:LikelihoodBoundOriginal}.  The assumed
inequality~(\ref{ScoreBoundEq}) implies that
\begin{align*}
  s_J(\beta_0)^{\top}\gamma&=\left(H_J(\beta_0)^{-\frac{1}{2}}s_J(\beta_0)\right)^{\top}\left(H_J(\beta_0)^{\frac{1}{2}}\gamma\right)\\
  &\le \sqrt{n c_{\mathrm{upper}}(\|\beta_0\|_2) } \cdot{\tau}_{|J\backslash
    J_0|} \norm{\gamma}_2.
\end{align*}
Therefore, for $\norm{\gamma}_2\leq \frac{
  c_{\mathrm{lower}}(\|\beta_0\|_2) }{c_{\mathrm{change}}}$, the
inequality~(\ref{eqn:LikelihoodBoundOriginal}) with another application
of~(\ref{Hess1}) gives
\begin{multline*}
\ln(\beta_0+\gamma)-\ln(\beta_0)
\leq \sqrt{n c_{\mathrm{upper}}(\|\beta_0\|_2) } \cdot{\tau}_{|J\backslash
    J_0|} \norm{\gamma}_2 \\
 -\frac{n c_{\mathrm{lower}}(\|\beta_0\|_2) (1 -   c_{\mathrm{lower}}(\|\beta_0\|_2) ^{-1}c_{\mathrm{change}}\cdot \norm{\gamma}_2)}{2}\norm{\gamma}_2^2 \;.
\end{multline*}
In particular, for $\norm{\gamma}_2\leq \frac{
  c_{\mathrm{lower}}(\|\beta_0\|_2) }{2c_{\mathrm{change}}}$, we have 
\begin{multline*}
\ln(\beta_0+\gamma)-\ln(\beta_0)\\
\leq \norm{\gamma}_2\left(\sqrt{n
    c_{\mathrm{upper}}(\|\beta_0\|_2) }\cdot {\tau}_{|J\backslash
    J_0|} -\frac{n c_{\mathrm{lower}}(\|\beta_0\|_2)
  }{4}\norm{\gamma}_2 \right)\;,
\end{multline*}
and so by concavity of the log-likelihood function, for all $\gamma\in\R^J$,
\begin{multline}
  \label{eq:concavity}
  \ln(\beta_0+\gamma)-\ln(\beta_0)
  \leq \norm{\gamma}_2\left(\sqrt{n
      c_{\mathrm{upper}}(\|\beta_0\|_2) }\cdot {\tau}_{|J\backslash
      J_0|} \right.\\
  \left.-\frac{n c_{\mathrm{lower}}(\|\beta_0\|_2)
    }{4}\min\left\{\norm{\gamma}_2,\frac{
        c_{\mathrm{lower}}(\|\beta_0\|_2)
      }{2c_{\mathrm{change}}}\right\} \right)\;.
\end{multline}
Since $\ln(\hat\beta_J)-\ln(\beta_0)\ge 0$, this shows  that
\[
\norm{\wh{\beta}_J-\beta_0}_2\;\leq\; \frac{4\sqrt{
    c_{\mathrm{upper}}(\|\beta_0\|_2) }\cdot {\tau}_{|J\backslash
    J_0|}}{\sqrt{n} c_{\mathrm{lower}}(\|\beta_0\|_2) }\;,
\]
as long as we assume that 
\[
\frac{4\sqrt{ c_{\mathrm{upper}}(\|\beta_0\|_2) }\cdot
  {\tau}_{|J\backslash J_0|}}{\sqrt{n}
  c_{\mathrm{lower}}(\|\beta_0\|_2) } \leq \frac{
  c_{\mathrm{lower}}(\|\beta_0\|_2) }{2c_{\mathrm{change}}}\;.
\]

Taking up~(\ref{eq:score:bound:diff:maxima}), we get
\begin{multline*}
  \ln(\wh{\beta}_J)-\ln(\wh{\beta}_{J_0})\
  \leq \frac{1}{2}\tilde{\tau}_{|J\backslash J_0|}^2 \\
  + n\cdot c_{\mathrm{change}}\left(\frac{c_{\mathrm{upper}}(\|\beta_0\|_2)}{
      c_{\mathrm{lower}}(\|\beta_0\|_2)}\right)^{\frac{3}{2}} 
  \left(\frac{4\sqrt{ c_{\mathrm{upper}}(\|\beta_0\|_2) }\cdot
      {\tau}_{|J\backslash J_0|}}{\sqrt{n}
      c_{\mathrm{lower}}(\|\beta_0\|_2) }\right)^3\;.
\end{multline*}
If
\begin{equation}\label{eqn:n_lower_bd_step}
\sqrt{n} \geq 
\frac{2c_{\mathrm{change}}}{\eps'}\left(\frac{c_{\mathrm{upper}}(\|\beta_0\|_2)}{
      c_{\mathrm{lower}}(\|\beta_0\|_2)}\right)^{\frac{3}{2}} 
      \left(\frac{4\sqrt{ c_{\mathrm{upper}}(\|\beta_0\|_2)}}{  c_{\mathrm{lower}}(\|\beta_0\|_2)}\right)^3
\cdot \frac{\tau_{|J\backslash
  J_0|}^3}{\tilde{\tau}_{|J\backslash J_0|}^2}\;,
  \end{equation}
  then we get
\begin{align}
  \nonumber
  \ln(\wh{\beta}_J)-\ln(\wh{\beta}_{J_0})
  &\leq \frac{1}{2}\tilde{\tau}_{|J\backslash J_0|}^2\cdot (1+\eps')
  \\
  \label{eq:choice:of:epsprime1}
  &=\frac{(1+\eps')}{(1-\eps')^3}\cdot \Big(|J\backslash
  J_0|\log\left(\frac{6p}{\eps'}\right) + \log(4p^{\nu})\Big).
\end{align}
Hence, this inequality holds whenever \eqref{eqn:n_lower_bd_step}
holds.  Now, to determine a simpler lower bound on $n$, we calculate
\[
\frac{\tau_{|J\backslash J_0|}^2}{\tilde{\tau}_{|J\backslash
    J_0|}^2}=\frac{{\frac{2}{(1-\eps')^3}\cdot\left[{|J|\log\left(\frac{3}{\eps'}\right)
        + \log(4p^{\nu})+ |J\backslash J_0|\log(2p)}\right]}}
{{\frac{2}{(1-\eps')^3}\cdot\left[{|J\backslash
        J_0|\log\left(\frac{3}{\eps'}\right) + \log(4p^{\nu})+
        |J\backslash J_0|\log(2p)}\right]}} \leq
{\frac{|J|}{|J\backslash J_0|}}\leq q\;.\] Hence,
\eqref{eqn:n_lower_bd_step} holds as long as $n$ exceeds a constant
multiple of $q^2\tau_{|J\backslash J_0|}^2$.  For $p$ large enough,
which we can ensure by choice of the constant $C_\mathrm{dim}$, we
have that $\tau_{|J\backslash J_0|}^2$ is no larger than a constant
times $q\log(p)$.  So, by choosing the constant
$C_{\mathrm{sample},1}$ appropriately, \eqref{eqn:n_lower_bd_step}
holds as long as
\[n \geq C_{\mathrm{sample},1}\cdot q^3\log(p)\;.\]

 Now fix $\eps'\in(0,\eps)$ such that
\begin{equation}
  \label{eq:choice:of:epsprime2}
  \frac{(1+\eps')^2}{(1-\eps')^3}< 1+\eps.
\end{equation}
Choosing the constant $C_\mathrm{dim}$ to ensure that $p$ is large enough, we
have that 
\[
\frac{|J\backslash J_0|\log\left(\frac{6p}{\eps'}\right) +
  \log(4p^{\nu})}{(|J\backslash J_0|+\nu)\log(p)}\;=\;
1+\frac{|J\backslash J_0|\log\left(\frac{6}{\eps'}\right) +
  \log(4)}{(|J\backslash J_0|+\nu)\log(p)} \;\le\; 1+\eps',
\]
which implies, by~(\ref{eq:choice:of:epsprime1})
and~(\ref{eq:choice:of:epsprime2}), that
\begin{align*}
  \ln(\wh{\beta}_J)-\ln(\wh{\beta}_{J_0})
  &\leq  (1+\eps)(|J\backslash J_0|+\nu)\log(p)\;.
\end{align*}
This proves statement~\ref{thm2:a} of
Theorem~\ref{thm:HessianToLikelihood}.

To show the remaining claim~\ref{thm2:b} of
Theorem~\ref{thm:HessianToLikelihood}, we first note that for any
$J\not\supset J_0$, it holds that
\[
\norm{\wh{\beta}_J-\beta_0}_2\;\geq\;\min_{j\in J_0}|(\beta_0)_j|.
\]
Having assumed that the Hessian conditions hold for true models with
up to $2q$ covariates, we may apply~(\ref{eq:concavity}) to the
model given by $(J\cup J_0)\supseteq J$.  We deduce that
\begin{multline*}
\ln(\wh{\beta}_J)-\ln(\beta_0)\leq 
\min_{j\in J_0}|(\beta_0)_j|\left(\sqrt{n
    c_{\mathrm{upper}}(\|\beta_0\|_2) }\cdot {\tau}_{|J\backslash
    J_0|} \right.
\\
\left. -\frac{n c_{\mathrm{lower}}(\|\beta_0\|_2)
  }{4}\min\left\{\min_{j\in J_0}|(\beta_0)_j|,\frac{
      c_{\mathrm{lower}}(\|\beta_0\|_2)
    }{2c_{\mathrm{change}}}\right\} \right)\;,
\end{multline*}
as long as the term in the parentheses is non-positive.  However, this
can be guaranteed to be the case, by appropriate choice of the
constant $C_{\mathrm{sample},2}$. In particular, for appropriate choice
of $C_{\mathrm{sample},2}$, we get
\begin{multline*}
\ln(\wh{\beta}_J)-\ln(\beta_0)\leq \\
 - \min_{j\in J_0}|(\beta_0)_j| \frac{n c_{\mathrm{lower}}(\|\beta_0\|_2)
  }{8}\min\left\{\min_{j\in J_0}|(\beta_0)_j|,\frac{
      c_{\mathrm{lower}}(\|\beta_0\|_2)
    }{2c_{\mathrm{change}}}\right\}\;.
\end{multline*}
Since  $\min_{j\in J_0}|(\beta_0)_j|$ is also upper bounded by a constant, namely
$\min_{j\in J_0}|(\beta_0)_j|\leq \|\beta_0\|_2\leq a_0$,
this is sufficient to prove claim~\ref{thm2:b} of
Theorem~\ref{thm:HessianToLikelihood}.  

\section{Proof of Hessian conditions (Theorem~\ref{thm:RandomToHessian})}\label{sec:Hessian}

This part of the appendix provides the proof of
Theorem~\ref{thm:RandomToHessian}, according to which the
assumptions~\ref{ass:A1}-\ref{ass:A3} from
Section~\ref{sec:logist-reg} yield a well-behaved Hessian matrix for
the log-likelihood function of all sparse submodels of a logistic
regression model.  The proof is split into three parts.  First, we
address the inequality~(\ref{Hess2}), next the upper bound in
(\ref{Hess1}) and then the lower bound in~(\ref{Hess1}).
In each case we provide an explicit probability for an event that
ensures the desired conclusion.  A union bound over the three cases
implies that all inequalities hold simultaneously with a probability
large enough to conform with the assertion of
Theorem~\ref{thm:RandomToHessian}.

\subsection{Upper bound on change in Hessian}

Define the constant
\begin{equation}
  \label{eq:constant-c}
  c_{\mathrm{change}}= 1+a_2+12\sqrt{2}\, a_3^3.
\end{equation}
We claim that if $n\geq q^3\log(2p)$, then with
probability at least
\begin{equation}
  \label{eq:prob:Hess2}
1-\exp\left\{-\frac{n}{2a_3^6q^3}\right\}\;,
\end{equation}
we have
\[
\sup_{J\subseteq [p], |J|\leq q}\sup_{\beta\neq
  \beta'\in\R^J}\frac{\Norm{H_J(\beta)-H_J(\beta')}}{\norm{\beta-\beta'}_2}\leq
c_{\mathrm{change}}\cdot n\;.
\]

To show this claim, take any set $J$ with $|J|\leq q$, any unit vector
$u\in\R^J$ and any pair of distinct vectors $\beta\neq \beta'\in\R^J$.
 Then we have
\begin{multline*}
\left|u^{\top}\left(H(\beta) - H(\beta')\right) u \right|
\;\leq\;  \sum_{i=1}^n (X_i^{\top}u)^2\cdot
\left|\mathrm{b}''(X_i^{\top}\beta) -
  \mathrm{b}''(X_i^{\top}\beta')\right|\\ 
\;\leq\;  \sum_{i=1}^n (X_i^{\top}u)^2\cdot \left|X_i^{\top}\beta -
  X_i^{\top}\beta'\right|\cdot\max_{t\in[0,1]}\left|\mathrm{b}'''(X_i^{\top}(t\beta+(1-t)\beta')\right|. 
\end{multline*}
Define the unit vector $v =
\frac{\beta-\beta'}{\norm{\beta-\beta'}_2}\in\R^J$.  In logistic
regression, $|\mathrm{b}'''(\theta)|\leq 1$ for all $\theta$.  Using
this fact\footnote{For other exponential families, one could bound the
  $\mathrm{b}'''(\cdot)$ term by taking $q$ to be constant and only
  considering $\beta$ and $\beta'$ of bounded norm.}, we obtain that
\begin{align*}
\left|u^{\top}\left(H(\beta) - H(\beta')\right) u \right| &\leq
  \norm{\beta-\beta'}_2\cdot n \left(\frac{1}{n} \sum_{i=1}^n 
  (X_i^{\top}u)^2\cdot
  \left|X_i^{\top}v\right|\right)\\
&\leq  \norm{\beta-\beta'}_2 \cdot n \left(\frac{1}{n} \sum_{i=1}^n
  |X_i^{\top}u|^3\right)^{\frac{2}{3}}\left(\frac{1}{n}
  \sum_{i=1}^n |X_i^{\top}v|^3\right)^{\frac{1}{3}}\\ 
&\leq  \norm{\beta-\beta'}_2 \cdot n \cdot \left(
  \sup_{\text{$q$-sparse unit $w$}} \frac{1}{n}\sum_{i=1}^n
  |X_i^{\top}w|^3\right)\;. 
\end{align*}
Applying Corollary \ref{cor:MomentsSums} for exponent $k=3$, we find
that with at least the claimed probability from~(\ref{eq:prob:Hess2}),
\[\sup_{\text{$q$-sparse unit $w$}} \frac{1}{n}\sum_{i=1}^n |X_i^{\top}w|^3
\; \leq\;
 1+a_2+12\sqrt{2}\, a_3^3 = c_{\mathrm{change}}\;,\]
as long as $n\geq q^3\log(2p)$.  
Since $H(\beta)-H(\beta')$ is symmetric, this implies that
\begin{equation*}
  \frac{\Norm{H_J(\beta)-H_J(\beta')}}{\norm{\beta-\beta'}_2}
  \;\le\; \sup_{\text{$q$-sparse unit
      $u$}}\frac{\left|u^{\top}(H(\beta)-H(\beta'))u\right|}{\norm{\beta-\beta'}_2}
  \;\leq\;
  c_{\mathrm{change}} \cdot n
\end{equation*}
for all sets $J$ of cardinality $|J|\leq q$ and all
$\beta\neq\beta'\in\R^J$, as claimed.

\subsection{Upper bound on Hessian}

In this subsection, we prove that if inequality~\eqref{Hess2} holds,
then with probability at least 
\begin{equation}
  \label{eq:prob:Hess1:upper}
1-\exp\left\{-\frac{n}{2a_3^4q^2}\right\} ,
\end{equation}
it also holds that 
\[H_J(\beta)\preceq n\cdot
c_{\mathrm{upper}}(\norm{\beta}_2)\cdot\mathbf{I}_J\;\] for all
$J$ with $|J|\leq q$ and all $\beta\in\R^J$.  Here, we
define
\[c_{\mathrm{upper}}(r)\coloneqq\mathrm{b}''(0)\cdot\left( 1 + a_2 +
  8\sqrt{2}\,a_3^2\right)+ c_{\mathrm{change}}\cdot r\;\] where
$c_{\mathrm{change}}$ is the constant from~(\ref{eq:constant-c}) and
$\mathrm{b}''(0)=1/4$ for logistic regression.  The idea for our proof
is to show that, on a suitable event, $\sup_{|J|\leq
  q}\Norm{H_J(0)}=\O{n}$.  Then, combined with the bounded change
condition \eqref{Hess2}, we will be able to bound $\Norm{H_J(\beta)}$
for any $\beta\in\R^J$.

First, for any $q$-sparse unit $u$, we have
\[\EE{(X^{\top}u)^2}\leq \EE{|X^{\top}u|^3}^{\frac{2}{3}}\leq a_2^{\frac{2}{3}}\;.\]
Then we have, with at least the probability in~(\ref{eq:prob:Hess1:upper}),
\begin{align*}
  \sup_{|J|\leq q}\Norm{H_J(0)}
  &=\sup_{|J|\leq q}\Norm{\sum_{i=1}^n
    X_{iJ}X_{iJ}^{\top}\mathrm{b}''(X_i^{\top}0)}\\
  &= \mathrm{b}''(0)\cdot  \sup_{|J|\leq q}\Norm{\sum_{i=1}^n
    X_{iJ}X_{iJ}^{\top}}\\
  &=\mathrm{b}''(0)\cdot\sup_{|J|\leq q,\text{ unit }u\in\R^J}\sum_{i=1}^n (X_i^{\top}u)^2\\
  & \leq \mathrm{b}''(0)\cdot  n\left( 1 + a_2 +  8\sqrt{2}\,a_3^2\right)\;,
\end{align*}
where for the last step we apply Corollary \ref{cor:MomentsSums} with
$k=2$, using the assumption that $n\geq q^2\log(2p)$.  The
bounded change condition from \eqref{Hess2} now implies the desired
conclusion, namely, that for all $J$ with $|J|\leq q$ and
all $\beta\in\R^J$,
\[
\norm{H_J(\beta)}\leq \norm{H_J(0)}+
\norm{H_J(0)-H_J(\beta)}\leq n\cdot c_{\mathrm{upper}}(\norm{\beta}_2)\;.
\]

\subsection{Lower bound on Hessian}

Finally, we prove that with probability at least 
\begin{equation}
  \label{eq:prob:Hess1:lower}
  1-2\exp\left\{-\frac{n}{2}\cdot
    \left(\frac{a_1^3}{512a_2^2}\right)^2\right\},
\end{equation}
it holds for all $|J|\leq q$, for all $|J|\leq q$ and for all $\beta\in\R^J$ that
\[
H_J(\beta)\succeq n\cdot c_{\mathrm{lower}}(\norm{\beta}_2)\cdot\mathbf{I}_J\;,
\]
where
\begin{align*}
  c_{\mathrm{lower}}(r) &\coloneqq   \frac{ a_1^4}{2048 a_2^2}\cdot
  \min_{|\theta|\leq r\cdot 2\sqrt[3]{256}
  a_2/a_1}\mathrm{b}''(\theta)\\ 
  &= \frac{ a_1^4}{2048 a_2^2}\cdot \frac{\exp\{ r\cdot 2\sqrt[3]{256}
    a_2/a_1\}}{\left(1+\exp\{r\cdot 2\sqrt[3]{256} a_2/a_1
      \}\right)^2}
\end{align*}
for the case of logistic regression; recall~(\ref{eq:b':b''}).  We
show this for triples $(n,p,q)$ that have $n$ larger than
the product of $q\log(2p)$ and a constant that is determined
through~(\ref{eq:other-sample-size-requirement}) below.

For a proof, since 
\[
H_J(\beta)=\sum_{i=1}^n X_{iJ}X_{iJ}^{\top}\,\mathrm{b}''(X_i^{\top}\beta),
\]
we consider the quantity 
\[
\sum_{i=1}^n (X_i^{\top}u)^2\mathrm{b}''(X_i^{\top}\beta)
\]
 where $u\in\R^J$ is a unit vector.   For any choice of $w_1,w_2\geq 0$, we have
\begin{equation*}
  \sum_{i=1}^n (X_i^{\top}u)^2\mathrm{b}''(X_i^{\top}\beta)\; \geq\;
  \sum_{i=1}^n  w_1^2\, \min_{|\theta|\leq \norm{\beta}_2 w_2
  }\mathrm{b}''(\theta)\cdot \one{\left\{|X_i^{\top}u|\geq w_1,
  |X_i^{\top}\beta|\leq \norm{\beta}_2  w_2\right\}}.
\end{equation*}
Using the symmetry and monotonicity of $\mathrm{b}''$ for logistic
regression we find 
\begin{multline}
  \label{StartLowerBound}
  \sum_{i=1}^n (X_i^{\top}u)^2\mathrm{b}''(X_i^{\top}\beta) \; \geq\;
  n\cdot w_1^2\,\mathrm{b}''(\norm{\beta}_2  w_2) 
  \\
  \times    
\left(1 -  \frac{\#\{i:|X_i^{\top}u|<w_1\}}{n} -
  \frac{\#\left\{i:|X_i^{\top}\beta|/\norm{\beta}_2>w_2\right\}}{n}\right)\;. 
\end{multline}
We now show how to choose $w_1$ and $w_2$ such that the two relative
frequencies are sufficiently small, with high probability, for any
choice of $u$ and $\beta$.

By Lemma \ref{IntervalLemma}, for any $t>0$, with probability at least
$1-2e^{-nt^2/2}$, for all $q$-sparse unit vectors $u$,
\begin{equation}
\label{aProb}
\frac{\#\{i:|X_i^{\top}u|<w_1\}}{n}\;\leq\;
\pr\left\{|X_1^{\top}u|<w_1 + \frac{1}{t}
  \sqrt{\frac{32a_3^2q\log(2p)}{n}}\right\} + 2t 
\end{equation}
and
\begin{equation}
\label{bProb}
\frac{\#\{i:|X_i^{\top}u|>w_2\}}{n}\;\leq\;
\pr\left\{|X_1^{\top}u|>w_2 - \frac{1}{t}
  \sqrt{\frac{32a_3^2q\log(2p)}{n}}\right\} + 2t\;. 
\end{equation}
Now set 
\begin{align*}
  w_1&=\sqrt{\frac{ a_1}{8}}, &
  w_2&=\frac{2\sqrt[3]{256} a_2}{ a_1}, &
  t&=\frac{ a_1^3}{512 a_2^2},
\end{align*}
and assume 
\begin{equation}
  \label{eq:other-sample-size-requirement}
  \frac{1}{t}\sqrt{\frac{32a_3^2q\log(2p)}{n}} \;\leq\;
  \min\left\{\sqrt{\frac{ a_1}{8}}, \frac{\sqrt[3]{256} a_2}{
      a_1}\right\}\;. 
\end{equation}
Then, for shorter notation, define the two scalars
\begin{align*}
  w_1' &= w_1+\frac{1}{t}\sqrt{\frac{32a_3^2q\log(2p)}{n}}, &
  w_2' = w_2 - \frac{1}{t}\sqrt{\frac{32a_3^2q\log(2p)}{n}}\;.
\end{align*}

Assume now that \eqref{aProb} and \eqref{bProb} hold.
We begin by simplifying the bound in  \eqref{aProb}.
By~(\ref{eq:other-sample-size-requirement}), $w_1'{}^2\leq
\frac{a_1}{2}$, and so applying Lemma \ref{AwayFromZero} with
$Z=(X_1^{\top} u)^2$, $h(Z)=\sqrt{Z}$ and $a=w_1'{}^2$ yields that for
all $q$-sparse unit vectors $u$,
\begin{align*}
  \pr\left\{|X_1^{\top}u|<w_1'\right\}  
  &\leq 1 - \frac{\EE{(X_1^\top u)^2} - w_1'{}^2}{\inf\left\{x\geq 0 : \sqrt{x}>\frac{\EE{|X_1^\top u|^3}}{\EE{(X_1^\top u)^2}-w_1'{}^2}\right\}}\\
  &\leq 1 - \frac{ a_1-w_1'{}^2}{\inf\left\{x\geq 0 : \sqrt{x}>\frac{\EE{|X_1^\top u|^3}}{\EE{(X_1^\top u)^2}-w_1'{}^2}\right\}}.
\end{align*}
The term involving the supremum satisfies
\begin{equation*}
  \frac{ a_1-w_1'{}^2}{\inf\left\{x\geq 0:\sqrt{x}>\frac{ a_2}{ a_1-w_1'{}^2}\right\}} 
  =  \frac{ a_1-w_1'{}^2}{\left(\frac{ a_2}{ a_1-w_1'{}^2}\right)^2} 
  =  \frac{( a_1-w_1'{}^2)^3}{8 a_2^2} \geq \frac{ a_1^3}{64
  a_2^2}\,,
\end{equation*}
and so
\[
\sup_{\text{$q$-sparse unit $u$}}  \frac{\#\{i:|X_i^{\top}u|<w_1\}}{n}\;\leq\;
\left(1- \frac{ a_1^3}{64 a_2^2}\right) + 2t 
\;\leq\;  1- \frac{3
    a_1^3}{256 a_2^2}\;.
\]
Next, we simplify the  bound in \eqref{bProb}.
By~(\ref{eq:other-sample-size-requirement}), for any $u$,
\begin{equation*}
  \pr\left\{|X_1^{\top}u|>w_2'\right\}
\leq \pr\left\{|X_1^{\top}u|>\frac{\sqrt[3]{256} a_2}{ a_1}\right\}
= \pr\left\{|X_1^{\top}u|^3>\frac{256 a_2^3}{ a_1^3}\right\}.
\end{equation*}
By Markov's inequality,
\begin{equation*}
  \pr\left\{|X_1^{\top}u|>w_2'\right\}
\leq\frac{\EE{|X_1^{\top}u|^3}}{{256 a_2^3}/{ a_1^3}}
\leq\frac{ a_2}{{256 a_2^3}/{ a_1^3}}
= \frac{ a_1^3}{256 a_2^2}\;.
\end{equation*}
We obtain that
\[
\sup_{\text{$q$-sparse unit $u$}} \frac{\#\{i:|X_i^{\top}u|>w_2\}}{n} \;\leq\;
 \frac{ a_1^3}{256 a_2^2}+2t \;\leq\;  \frac{ a_1^3}{128
  a_2^2}\;.
\]

Returning to \eqref{StartLowerBound}, we conclude that, with at least
the probability from~(\ref{eq:prob:Hess1:lower}), for all $|J|\leq q$
and all unit $u\in\R^J$ and all $\beta\in\R^J$, 
\begin{align*}
\sum_{i=1}^n (X_i^{\top}u)^2\,\mathrm{b}''(X_i^{\top}\beta) &\geq
n\cdot w_1^2 \mathrm{b}''(\norm{\beta}_2 w_2 )\cdot \left[1 -  \left(1- \frac{3
    a_1^3}{256 a_2^2} \right) - \frac{ a_1^3}{128
  a_2^2}\}\right]\\
&\geq n\cdot\frac{ a_1}{8}\cdot \frac{ a_1^3}{256
  a_2^2}\, \mathrm{b}''\left(\norm{\beta}_2\cdot
  \frac{2\sqrt[3]{256} a_2}{ a_1}\right)\\
&= n\cdot c_{\mathrm{lower}}(\norm{\beta}_2)\;.
\end{align*}

\section{Technical lemmas}
\label{sec:technical-lemmas}

This section of the appendix provides the lemmas that were used in
previous parts of the paper to control the behavior of sparse linear
combinations of the covariates.

\subsection{Concentration bound and subgaussian maxima}

The lemmas we establish subsequently make use of the following general
concentration bound.

\begin{lemma}
  \label{lem:kolt}
  Let $X,X_1,\dots,X_n$ be i.i.d.~random variables drawn from a set
  $\mathcal{X}$, and let $\mathcal{F}$ be a class of functions
  $f:\mathcal{X}\rightarrow\R$.  Consider an $L$-Lipschitz function
  $g:\R\rightarrow\R$ with $g(0)=0$ and $|g(f(X))|\leq M$ almost
  surely.  Then, for any $t\geq 0$, with probability at least
  $1-e^{-t^2/2}$,
  \begin{multline*}
    \label{ConcKolt}
    \sup_{f\in\mathcal{F}}\left|\sum_{i=1}^n\Big( g(f(X_i)) -
        \Ep{X}{g(f(X))}\Big) \right|\leq\\
    4L\,\Ep{\nu_1,\dots,\nu_n,X_1,\dots,X_n}{\sup_{f\in\mathcal{F}}\left|\sum_{i=1}^n
        \nu_i f(X_i)\right|} + t\cdot M\sqrt{n}\;,
  \end{multline*} 
where $\nu_1,\dots,\nu_n\in\{\pm 1\}$ are independent Rademacher
random variables that are also independent of $X_1,\dots,X_n$.
\end{lemma}
\begin{proof}
  The claim follows by combining known bounded difference, symmetrization,
  and contraction results.  Specifically, it is a consequence of
  Theorems 2.5, 2.1, and 2.3 in \cite{koltchinskii:2011}.
\end{proof}

The next lemma states a well-known property of subgaussian random
variables \cite[Prop.~3.1]{koltchinskii:2011}. 
Recall that a random variable $Z$ is $\sigma$-subgaussian if, for all
$t\in\mathbb{R}$,
\[
\EE{e^{t Z}}\leq e^{t^2\sigma^2/2}.
\]

\begin{lemma}
  \label{lem:max:subgauss}
  Suppose $Z_1,\dots,Z_m$ are, not necessarily independent, random
  variables with a common distribution that is $\sigma$-subgaussian
  for $\sigma>0$.  Then  
  \[
  \EE{\max_{1\le i\le m} |Z_i|} \;\le\; \sigma \cdot \sqrt{2\log(2m)}.
  \]  
\end{lemma}

\subsection{Sparse unit linear combinations falling in an interval}

We return to the setting where $X,X_1,\dots,X_n$ are i.i.d.~random
vectors in $\R^p$ and satisfy assumptions~\ref{ass:A1}-\ref{ass:A3}.

\begin{lemma}\label{IntervalLemma}
Fix any $a>0$ and $t>0$. With probability at least $1-e^{-t^2n/2}$,
for all $q$-sparse unit vectors $u$,
\[
\frac{1}{n}\#\{i:|X_i^{\top}u|<a\}\;\leq\; \pr\left\{|X^{\top}u|<a +
  \frac{1}{t} \sqrt{\frac{32a_3^2q\log(2p)}{n}}\right\} + 2t\;. 
\]
Similarly, with probability at least $1-e^{-t^2n/2}$, for all
$q$-sparse unit vectors $u$,
\[
\frac{1}{n}\#\{i:|X_i^{\top}u|>a\}\;\leq\; \pr\left\{|X^{\top}u|>a -
  \frac{1}{t} \sqrt{\frac{32a_3^2q\log(2p)}{n}}\right\} + 2t\;. 
\] 
\end{lemma}
\begin{proof}
The proofs of the two statements are essentially identical, so we
prove only the first one. Let
\[
c\coloneqq t^{-1}
\sqrt{\frac{32a_3^2q\log(2p)}{n}},
\]
 and define the piece-wise linear function
\[
g(z)=
\begin{cases}
  1& \text{if } |z|\leq a,\\
  0& \text{if } |z|\geq a+c,\\
  (a+c-|z|)/c& \text{if } a\leq|z|\leq a+c.
\end{cases}
\]
Then $g$ is $1/c$-Lipschitz, has values in $[0,1]$, and satisfies
\[
\one{\{|z|<a\}}\leq g(z)\leq \one{\{|z|< a+c\}}.
\]
By the concentration bound from Lemma~\ref{lem:kolt},
with probability at least $1-e^{-t^2n/2}$,
\begin{multline}
  \label{gEvent}
  \sup_{\text{$q$-sparse unit $u$}}  \left|\sum_{i=1}^n \Big(g(X_i^{\top}u
    ) - \EE{g(X^{\top}u )}\Big)\right| \leq \\
  4c^{-1}\Ep{\nu,X}{\sup_{\text{$q$-sparse unit $u$}}
    \left|\sum_{i=1}^n \nu_i X_i^{\top}u \right|}+ nt\;.
\end{multline} 
Assume from now on that this event occurs.  

We may bound the expectation appearing in~(\ref{gEvent}) as
\begin{multline}
 \label{eq:bound:exp-sup}
  \Ep{\nu,X}{\sup_{\text{$q$-sparse unit $u$}} \left|\sum_{i=1}^n
      \nu_i X_i^{\top}u \right|}
  \leq \Ep{\nu,X}{\sup_{|J|\leq q}\Norm{\sum_{i=1}^n \nu_i
      X_{iJ}}_2}\\
  \leq \sqrt{q}\,\Ep{\nu,X}{\sup_{1\le j\le p} \left|\sum_{i=1}^n \nu_i
      X_{ij}\right|}.
\end{multline}
By assumption~\ref{ass:A3}, $|X_{ij}|\leq a_3$ for all $i,j$, and therefore
$\sum_{i=1}^n \nu_i X_{ij}$ is $(a_3\sqrt{n})$-subgaussian
because, by independence of the $X_i$'s,
\begin{multline*}
\EE{e^{t\sum_{i=1}^n \nu_iX_{ij}}}=\prod_{i=1}^n \EE{e^{t \nu_iX_{ij}}}\leq \prod_{i=1}^n \EE{\EE{e^{t \nu_iX_{ij}}\mid X_{ij}}}\\
=\prod_{i=1}^n \frac{1}{2}\EE{e^{t X_{ij}} + e^{-t X_{ij}}} \leq \prod_{i=1}^n \EE{e^{t^2X_{ij}^2/2}}
\leq e^{nt^2a_3^2/2}.
\end{multline*}
Applying Lemma~\ref{lem:max:subgauss} 
to~\eqref{eq:bound:exp-sup}, we obtain that
\begin{equation}
  \label{eq:bound:exp-sup-subgauss}
  \Ep{\nu,X}{\sup_{\text{$q$-sparse unit $u$}} \left|\sum_{i=1}^n
      \nu_i X_i^{\top}u \right|}\;\leq\;\sqrt{q}\cdot a_3\sqrt{n}\cdot\sqrt{2\log(2p)}\;.
\end{equation}
We deduce that
\begin{multline*}
\sup_{\text{$q$-sparse unit $u$}}  \left|\sum_{i=1}^n g(X_i^{\top}u ) - \EE{g(X^{\top}u )}\right| \\
        \;\leq \;
         4c^{-1}\sqrt{q}\cdot a_3\sqrt{n}\cdot\sqrt{2\log(2p)}+ nt\;=\;2nt\;,
 \end{multline*}        
 by our choice of $c$.  Returning to \eqref{gEvent}, we have shown
 that, as desired,
\begin{multline*}
\sup_{\text{$q$-sparse unit $u$}}\#\{i:|X_i^{\top}u|<a\}\;\leq\; \sup_{\text{$q$-sparse unit $u$}} \sum_{i=1}^n g(X_i^{\top}u) 
\;\leq\;\\
 n\EE{g(X^{\top}u)} + 2nt \;\leq\; n\pr\{|X^{\top}u|<a + c\}+2nt\;.
\end{multline*}
\end{proof}

\subsection{Bounding functions of sparse unit linear combinations}

\begin{lemma}\label{lem:ExpectationsToSums}
  Suppose $f:[0,\infty)\rightarrow[0,\infty)$ is a nondecreasing
  function with
  \[
    f(z)\le M,\quad  |f(z)-f(z')|\le L |z-z'|,
\]
  for all\ $0\le z,z'\le a_3\sqrt{q}$.  If
  $\EE{f(|X^{\top}u|)}\leq a_2$ for all $q$-sparse unit vectors $u$,
  then under assumption \ref{ass:A3},
   it holds with probability at least $1-e^{-n/(2M^2)}$ that
  \[
  \sup_{\text{$q$-sparse unit $u$}} \sum_{i=1}^n   f(|X_i^{\top}u|)
  \;\leq\; n\left( 1 + a_2 +  4L a_3\sqrt{\frac{2\log(2p)}{n}}\right)\;.
  \]
\end{lemma}
\begin{proof}
  Define $h(z)=f(|z|)$ for $z\in\R$.  By our assumptions on $f$, the
  function $h$ is $M$-bounded and $L$-Lipschitz over $z\in[-a_3\sqrt{q},a_3\sqrt{q}]\subset\R$.
  
  Applying the concentration bound from Lemma~\ref{lem:kolt} to $h$,
  with $t=\sqrt{n}/M$, we obtain that with probability at least
  $1-e^{-n/(2M^2)}$,
  \begin{align*}
  \sup_{|J|\leq q, \text{ unit }u\in\R^J}&\left|\sum_{i=1}^n \Big(
    h(X_i^{\top}u)-\EE{h(X^{\top}u)}\Big)\right|\\
  &\leq
  n+4L\Ep{\nu,X}{\sup_{|J|\leq q, \text{
        unit }u\in\R^J}\left|\sum_{i=1}^n \nu_i X_i^{\top}u\right|}\\
  &\leq n + 4L \sqrt{q}\cdot a_3\sqrt{n}\cdot\sqrt{2\log(2p)} \;,
  \end{align*}
  where the second inequality follows
  from~(\ref{eq:bound:exp-sup-subgauss}) using the same reasoning
  as in the proof of Lemma~\ref{IntervalLemma}.
    Since $h(z)= f(|z|)$ for
  all $z\in\R$, we have 
  \[
  \EE{h(X^{\top}u)} =
  \EE{f(|X_i^{\top}u|)}\leq a_2.
  \]
  Hence, with probability at least $1-e^{-n/(2M^2)}$,
  \begin{equation}
    \label{eq:bound:h}
  \sup_{|J|\leq q, \text{ unit }u\in\R^J}\sum_{i=1}^n
  h(X_i^{\top}u)\;\leq \;
  n\left( 1 + a_2 +
    4La_3\sqrt{\frac{2q\log(2p)}{n}}\right)\;.
  \end{equation}
\end{proof}

We obtain the following corollary about moments of sparse unit linear
combinations. 

\begin{corollary}
  \label{cor:MomentsSums}
  Let $k>0$. If $\EE{|X^{\top}u|^k}\leq a_2$ for all
  $q$-sparse unit vectors $u$, then 
  \[ \sup_{|J|\leq q, \text{ unit }u\in\R^J}\sum_{i=1}^n
  |X_i^{\top}u|^k
  \leq n\left( 1 + a_2 +
    4\sqrt{2}\cdot k a_3^k\sqrt{\frac{q^k\log(2p)}{n}}\right)\;
  \]
  holds with probability at least
  \[
  1-\exp\left\{-\frac{n}{2a_3^{2k}q^k}\right\}.
  \]
\end{corollary}
\begin{proof}
Apply Lemma \ref{lem:ExpectationsToSums} to $f(|z|)=|z|^k$, setting
\begin{align*}
  M&=f\left(a_3\sqrt{q}\right)=\left(a_3\sqrt{q}\right)^k,
  \\ 
  L&=f'\left(a_3\sqrt{q}\right)=
  k\left(a_3\sqrt{q}\right)^{k-1},
\end{align*}
and collecting terms to find the upper bound.
\end{proof}

\subsection{Bounding a variable away from zero using expectations}

\begin{lemma}
  \label{AwayFromZero}
  Let $h:[0,\infty)\rightarrow[0,\infty)$ be a continuous and nondecreasing function
  such that $z\mapsto z\cdot h(z)$ is convex.
  Let $Z\geq 0$ be a random variable with $\EE{Z}<\infty$ and
  $\EE{Z\cdot h(Z)}<\infty$. Then for any $a\leq\EE{Z}$,
  \[\pr\left\{Z>a\right\}\geq \frac{\EE{Z}-a}{\inf\left\{x\geq 0 : h(x)>\frac{\EE{Z\cdot h(Z)}}{\EE{Z}-a}\right\}}\;.\]
\end{lemma}
\begin{remark} This inequality is an extension of the Paley-Zygmund inequality \cite[Inequality 4.9]{shorack:2000},
which makes the same statement for $h(z)=z$.\end{remark}
\begin{proof}
First, we have
\begin{equation}
  \label{eq:pz:lower}
  \EE{Z\cdot\one{Z>a}} = \EE{Z} - \EE{Z\cdot\one{Z\leq a}}\geq
  \EE{Z}-a\;.
\end{equation}
Now, let $Y$ be a random variable whose distribution is equal to the
distribution of $Z$ 
conditional on $Z>a$, that is,
\[\pr\left\{Y\leq y\right\} = \begin{cases}0,&y\leq a,\\ \frac{\pr\left\{a<Z\leq y\right\}}{\pr\left\{Z>a\right\}},&y>a.\end{cases}\]
Next, write $g(z)=z\cdot h(z)$, which by assumption is convex, continuous, and strictly increasing, with $g(0)=0$. We can then define
the inverse $g^{-1}:\R_+\mapsto\R_+$, which is also strictly increasing. Then,
by Jensen's inequality,
\begin{align*}
\EE{Z\cdot\one{Z>a}}
&=\EE{Y}\cdot\pr\left\{Z>a\right\}\\
&\leq g^{-1}\left(\EE{g(Y)}\right)\cdot\pr\left\{Z>a\right\}\\
&= g^{-1}\left(\EE{g(Z)\mid Z>a}\right)\cdot\pr\left\{Z>a\right\}\\
&= g^{-1}\left(\frac{\EE{g(Z)\cdot\one{Z>a}}}{\pr\left\{Z>a\right\}}\right)\cdot\pr\left\{Z>a\right\}\\
&\leq g^{-1}\left(\frac{\EE{g(Z)}}{\pr\left\{Z>a\right\}}\right)\cdot\pr\left\{Z>a\right\}
\end{align*}
and so
\[g^{-1}\left(\frac{\EE{g(Z)}}{\pr\left\{Z>a\right\}}\right) \geq
\frac{\EE{Z\cdot\one{Z>a}}}{\pr\left\{Z>a\right\}}\geq
\frac{\EE{Z}-a}{\pr\left\{Z>a\right\}},
\]
which implies
\[\frac{\EE{g(Z)}}{\pr\left\{Z>a\right\}} \geq g\left( \frac{\EE{Z}-a}{\pr\left\{Z>a\right\}}\right).\]
Rewriting this last conclusion in terms of $h$ gives
\[\frac{\EE{Z\cdot h(Z)}}{\pr\left\{Z>a\right\}}\geq
\frac{\EE{Z}-a}{\pr\left\{Z>a\right\}}\cdot h\left(
  \frac{\EE{Z}-a}{\pr\left\{Z>a\right\}}\right)\] 
and thus
\[\frac{\EE{Z\cdot h(Z)}}{\EE{Z}-a}\geq h\left( \frac{\EE{Z}-a}{\pr\left\{Z>a\right\}}\right)\;.\]
We conclude that for any $x$ such that
\[h(x)>\frac{\EE{Z\cdot h(Z)}}{\EE{Z}-a}\;,\] we must have $
\frac{\EE{Z}-a}{\pr\left\{Z>a\right\}}\leq x$ because $h$ is
nondecreasing. This proves that
\[\pr\left\{Z>a\right\}\geq \frac{\EE{Z}-a}{\inf\left\{x\geq 0 :
    h(x)>\frac{\EE{Z\cdot h(Z)}}{\EE{Z}-a}\right\}}\;.
\qedhere\]
\end{proof}

\bibliography{ising}

\end{document}